\documentclass[sn-mathphys,Numbered]{sn-jnl}

\usepackage{graphicx}%
\usepackage{multirow}%
\usepackage{amsmath,amssymb,amsfonts}%
\usepackage{amsthm}%
\usepackage{mathrsfs}%
\usepackage[title]{appendix}%
\usepackage{xcolor}%
\usepackage{textcomp}%
\usepackage{manyfoot}%
\usepackage{booktabs}%
\usepackage{algorithm}%
\usepackage{algorithmicx}%
\usepackage{algpseudocode}%
\usepackage{listings}%
\usepackage{mathtools}
\usepackage{comment}
\usepackage{enumitem}
\usepackage[short]{optidef}
\usepackage{pdflscape}
\usepackage[figuresright]{rotating}
\usepackage{pgfplots}
\usepackage{subcaption}
\usepackage{longtable}
\usepackage{csquotes}

\newtheorem{theorem}{Theorem}
\newtheorem{proposition}[theorem]{Proposition}%

\newtheorem{lemma}{Lemma}

\newtheorem{remark}{Remark}%

\DeclareMathOperator*{\argmax}{arg\,max}

\newcommand{\sign}{\textrm{sign}}
\newcommand{\tr}{\textrm{tr}}

\newcommand{\inprod}[2]{\langle #1,#2\rangle}

\pgfplotsset{compat=1.18}

\begin{document}

\title{Optimization meets Machine Learning: An Exact Algorithm for Semi-Supervised Support Vector Machines}


\author*[1]{\fnm{Veronica} \sur{Piccialli}}\email{veronica.piccialli@uniroma1.it}

\author[2]{\fnm{Jan} \sur{Schwiddessen}}\email{jan.schwiddessen@aau.at}

\author[1]{\fnm{Antonio M.} \sur{Sudoso}}\email{antoniomaria.sudoso@uniroma1.it}

\affil*[1]{\orgdiv{Department of Computer, Control and Management Engineering}, \orgname{Sapienza University of Rome}, \orgaddress{\street{Via Ariosto 25}, \city{Rome}, \postcode{00185}, \country{Italy}}}

\affil[2]{\orgdiv{Institut für Mathematik}, \orgname{Alpen-Adria-Universität Klagenfurt}, \orgaddress{\street{Universitätstraße 65-67}, \city{Klagenfurt}, \postcode{9020}, \country{Austria}}}


\abstract{Support vector machines (SVMs) are well-studied supervised learning models for binary classification. Large amounts of samples can be cheaply and easily obtained in many applications. What is often a costly and error-prone process is to label these data points manually. Semi-supervised support vector machines (S3VMs) extend the well-known SVM classifiers to the semi-supervised approach, aiming to maximize the margin between samples in the presence of unlabeled data. By leveraging both labeled and unlabeled data, S3VMs attempt to achieve better accuracy and robustness than traditional SVMs. Unfortunately, the resulting optimization problem is non-convex and hence difficult to solve exactly. This paper presents a new branch-and-cut approach for S3VMs using semidefinite programming (SDP) relaxations. We apply optimality-based bound tightening to bound the feasible set. Box constraints allow us to include valid inequalities, strengthening the lower bound. The resulting SDP relaxation provides bounds that are significantly stronger than the ones available in the literature. For the upper bound, instead, we define a local search heuristic exploiting the solution of the SDP relaxation. Computational results highlight the algorithm's efficiency, showing its capability to solve instances with ten times more data points than the ones solved in the literature.}

\keywords{Global Optimization, SVM, Branch-and-Cut, Semidefinite Programming}


\pacs[MSC Classification]{90C26,90C22,62H30}

\maketitle

\section{Introduction}

Support Vector Machines (SVMs) are a powerful class of supervised machine learning algorithms used for classification and regression tasks~\citep{cortes1995support}. SVMs are particularly effective in handling high-dimensional data and can be applied to both linearly separable and non-linearly separable datasets. The main idea behind SVMs is to find an optimal hyperplane that maximally separates the data points belonging to different classes in a feature space. This hyperplane is selected so that the margin, which is the distance between the hyperplane and the nearest data points from each class (called support vectors), is maximized. By maximizing the margin, SVMs aim to achieve better generalization and improve the ability to classify unseen data accurately. The SVM training problem can be equivalently formulated as a quadratic convex problem with linear constraints or, by Wolfe's duality theory, as a quadratic convex problem with one linear constraint and box constraints. Depending on the formulation, several optimization algorithms have been designed specifically for SVM training. We refer the reader to the survey~\cite{piccialli2018nonlinear} for an overview of the essential optimization methods for training SVMs.

Semi-supervised Support Vector Machines (S3VMs) are an extension of the traditional (supervised) SVMs that incorporate both labeled and unlabeled data during the training process~\citep{bennett1998semi}. Obtaining labeled data can be costly and time-consuming in many real-world applications, whereas unlabeled data are relatively easy to obtain in large quantities. When only a small percentage of the collected samples is labeled, learning a classification function is very likely to produce an inconsistent model, which would not be reliable for classifying future data. The motivation behind S3VMs is to exploit the unlabeled data to learn a decision boundary that is more robust and generalizable. In fact, by maximizing the margin in the presence of unlabeled data, the decision boundary produced by S3VMs traverses through low data-density regions while adhering to the labeled data points in the input space. In other words, this approach applies the cluster assumption to semi-supervised learning, which states that data points belonging to the same cluster have the same labels~\citep{chapelle2005semi, chapelle2009semi}. 

The design of SVMs with partially labeled datasets has been a highly active area of research. Much of this work revolves around a key concept: solving the conventional SVM problem while considering the unknown labels as extra optimization variables. However, the improved performance of S3VMs is obtained at the sacrifice of computational complexity. While the underlying optimization problem of SVMs is convex, this is not the case for S3VMs, where the problem becomes non-convex and has many low-quality solutions~\cite{chapelle2008optimization}. Since its introduction in~\cite{bennett1998semi}, a wide range of techniques has been employed to find good-quality solutions to the non-convex optimization problem associated with S3VMs~\citep{ding2017overview}. These techniques include combinatorial search~\citep{joachims1999transductive, burgard2023mixed}, gradient descent~\citep{chapelle2005semi}, continuation techniques~\citep{chapelle2006continuation}, convex-concave procedures~\citep{fung2001semi, collobert2006large}, semidefinite programming~\citep{de2006semi, bai2016conic}, nonsmooth optimization~\citep{astorino2007nonsmooth}, and deterministic annealing~\citep{sindhwani2006deterministic}. See the work by Chapelle et al.~\cite{chapelle2008optimization} and references therein for an overview of optimization procedures for S3VMs.


Many studies have demonstrated that S3VM implementations exhibit varying levels of empirical success. As shown in~\cite{chapelle2006branch}, this discrepancy is closely linked to their vulnerability to issues related to local optimizers. Empirical results on some semi-supervised tasks presented in~\cite{chapelle2006branch} show that the exact solution found by the branch-and-bound method has excellent generalization performance. In contrast, other S3VM implementations based on local optimization methods perform poorly. Local algorithms are designed to find solutions within a local region, and they can get trapped in low-quality local optima when dealing with non-convex problems. In the case of S3VMs, this can lead to a suboptimal hyperplane that does not generalize well to unseen data. 
Similar to most semi-supervised learning methods, there is no guarantee that S3VMs will outperform their supervised counterparts. This uncertainty stems from two potential scenarios: (a) the cluster assumption may not hold, and (b) the cluster assumption holds, but the presence of several local minima is problematic.

Chapelle et al.~\cite{chapelle2008optimization} conducted an empirical investigation comparing different S3VM implementations with a global optimizer. Their findings suggested that performance degradation is primarily due to suboptimal local minima, establishing a clear correlation between generalization performance and the S3VM objective function. Consequently, when the cluster assumption holds and hyperparameters are appropriately chosen, minimizing the S3VM objective is expected to be a suitable approach.
However, due to the non-convex nature of the problem, minimizing the S3VM objective is not a straightforward task, requiring the application of diverse optimization techniques, each with varying degrees of success. Moreover, the study by Chapelle et al.~\cite{chapelle2008optimization} revealed that no single heuristic method consistently outperforms others regarding generalization performance. This underscores the ongoing need for improved global optimization methods to optimize the S3VM objective function. Our paper addresses this gap by proposing an exact algorithm based on the branch-and-cut technique to achieve globally optimal solutions for S3VMs.

\section{Related Work}
This section reviews the exact approaches proposed in the literature to tackle S3VMs. Due to the complexity of the associated optimization problem, the computational time required for solving S3VM instances to global optimality quickly increases as the number of data points grows. This is why applications of S3VMs have been restricted to small-size problems. Most exact approaches in the literature are based on the branch-and-bound technique. A first branch-and-bound algorithm is proposed by Bennett and Demiriz~\cite{bennett1998semi}. The authors show that the L1-norm linear S3VM model can be reformulated as a Mixed-Integer Program (MIP) and solved exactly using integer programming tools. Instead, Chapelle et al.~\cite{chapelle2006branch} propose a branch-and-bound algorithm for L2-norm S3VMs. This algorithm performs a search over all possible labels and prunes large parts of the solution space based on lower bounds on the optimal S3VM objective. For the lower bound, they use the objective value of a standard SVM trained on the associated labeled set.
Next, they produce a binary enumeration tree where nodes are associated with a fixed partial labeling of the unlabeled dataset, and the two children correspond to the labeling of some new unlabeled data point. The root corresponds to the initial set of labeled examples, and the leaves correspond to a complete labeling of the data. Moving down in the tree, a labeling-confidence criterion is used to choose an unlabeled example on which to branch. The overall algorithm is able to return a globally optimal solution in a reasonable amount of time only for small-scale instances (less than 50 data points). Other approaches for solving S3VMs have been put forward. For instance, several studies have proposed convex relaxations of the objective function, which can be solved using Quadratic Programming (QP) and Semidefinite Programming (SDP) tools~\cite{bai2012sdp,bai2013new,bai2016conic}. {A semidefinite programming problem involves optimizing a linear objective function over the intersection of the cone of positive semidefinite matrices with an affine subspace. SDPs are convex optimization problems, and all well-posed SDPs can be solved in polynomial time, typically using interior-point methods. SDP relaxations are known to provide strong dual bounds for challenging combinatorial optimization problems. For a comprehensive overview of the theory, algorithms, and applications of semidefinite programming, we recommend referring to~\cite{wolkowicz2012handbook}.}

More recently, Tian and Luo~\citep{tian2017new} have proposed a new branch-and-bound algorithm for the L2-norm S3VM. First, they transform the original problem into a non-convex quadratically constrained quadratic programming (QCQP) problem. Then, to provide computationally efficient lower bounds, they apply Lagrangian relaxation to the QCQP model and dualize the non-convex quadratic constraints. The Lagrange multipliers are estimated by solving an auxiliary SDP and are then propagated downward to the child nodes during every branching step. At its core, the algorithm is a heuristic inspired by the branch-and-bound technique, involving a maximum number of branching decisions (labelings). For small problems, this decision count is fixed at 30, while for larger instances, the algorithm terminates the enumeration process when 50\% of unlabeled points have been labeled. It is worth highlighting that, in the general setting, this algorithm falls short of providing a certificate of global optimality. To the best of our knowledge, this is the most recent method based on the branch-and-bound technique available in the literature for L2-norm S3VMs. 

In this paper, we present a novel exact algorithm based on the branch-and-cut technique. {We enhance an existing SDP relaxation by employing sophisticated techniques from the global optimization literature. Specifically tailored to our problem, these methods yield significantly stronger bounds compared to those reported in existing literature.}
More in detail, the main original contributions of this paper are:
\begin{enumerate}
    \item We propose the first SDP-based branch-and-cut algorithm producing globally optimal solutions for S3VM. We perform bound tightening to derive optimality-based box constraints that allow us to strengthen the basic SDP relaxation available in the literature. We solve the resulting SDP relaxation through a cutting-plane procedure. We define an effective local search heuristic that exploits the solution of the SDP relaxation.
    \item We prove the tightness of the SDP relaxation for an ideal kernel function such that the S3VM objective function aligns with the ground truth.
    \item We perform an extensive set of numerical experiments on both synthetic and real-world instances. Our findings illustrate the algorithm's consistent ability to yield small optimality gaps for large-scale instances (up to 569 data points), a problem size approximately ten times larger than those addressed by existing exact algorithms in the literature.
    \item {We thoroughly analyze our results from both the mathematical optimization and machine learning perspectives, highlighting the advantages of using an exact algorithm.}
    \end{enumerate}


The remainder of this paper is organized as follows. In Section~\ref{sec:s3vm_intro}, we present the S3VM formulation and the convex relaxations available in the literature. Section~\ref{sec:bac} illustrates the ingredients of the proposed branch-and-cut approach, namely the SDP relaxation, the primal heuristic, and the branching strategy. In Section~\ref{sec:ideal}, we prove the tightness of the SDP relaxation under the assumption of the ideal kernel matrix. Section~\ref{sec:results} presents computational results on both synthetic and real-world instances. Finally, Section~\ref{sec:conclusion} concludes the paper with future research directions.

\subsection*{Notation}
Let $\mathbb{S}^n$ be the set of all $n\times n$ real symmetric matrices. We denote by $M\succeq 0$ a positive semidefinite matrix $M$ and by $\mathbb{S}_+^n$ the set of all positive semidefinite matrices of size $n\times n$. Analogously, we denote by $M \succ 0$ a positive definite matrix $M$ and by ${\mathbb S}_{++}^n$ the set of all positive definite matrices of size $n\times n$. We denote the trace inner product by $\inprod{\cdot}{\cdot}$. That is, for any $A, B \in \mathbb{R}^{n\times m}$, we define $\inprod{A}{B}\coloneqq \tr (B^\top A)$. Given a matrix $A \in \mathbb{R}^{n\times m}$, $A_{ij}$ denotes the element of $A$ in the $i$-th row and $j$-th column. For two vectors $a, b \in \mathbb{R}^n$, $a \circ b$ denotes the element-wise product between $a$ and $b$. Finally, we denote by $e_n$ the vector of all ones of size $n$ and by $I_n$ the identity matrix of size $n \times n$. We omit the subscript when the size is clear from the context.

\section{Semi-supervised Support Vector Machines}\label{sec:s3vm_intro}
In this section, we introduce the mathematical programming formulation behind S3VMs. We are given a binary classification dataset consisting of $l$ labeled points $\{(x_i, y_i)\}_{i=1}^l$ and $n-l$ unlabeled points $\{x_i\}_{i=l+1}^n$, where $x_i \in \mathbb{R}^d$ and $y_i \in \{-1, +1\}$ for $i \in \{1, \dots, n\}$. Similar to traditional SVMs, the primary objective of S3VMs is to determine the parameters $(w, b) \in \mathbb{R}^d \times \mathbb{R}$ of a separating hyperplane $w^\top x_i + b = 0$. This hyperplane separates all the data points in the input space into two classes by maximizing the distance between two supporting hyperplanes, namely $w^\top x_i + b = 1$ and $w^\top x_i + b = -1$, while ensuring accurate classification. The objective function to be optimized depends on both the decision boundary parameters $(w, b)$ and the unknown labels $\{y_i\}_{i=l+1}^n$:
\begin{equation}\label{eq:s3vm_loss}
    f(w, b, \{y_i\}_{l+1}^n) = \frac{1}{2} \|w\|^2_2 + C_l \sum_{i=1}^l  H(y_i, o_i) + C_u \sum_{i=l+1}^n H(y_i, o_i),
\end{equation}
where $o_i = w^\top x_i + b$ and $H$ is a loss function. The hyperparameters $C_l$ and $C_u$ are used to give more importance to the labeled and unlabeled error components, respectively. The Hinge loss, defined as
\begin{equation*}
    H(y_i, o_i) = \max(0, 1 - y_i o_i)^p,
\end{equation*}
is the most commonly used loss function and can be penalized either linearly ($p = 1$) or quadratically ($p = 2$). In the remainder
of this paper, we will consider $p = 2$ which gives rise to the L2-norm S3VM formulation. 

It is important to note that excluding the third term from~\eqref{eq:s3vm_loss} simplifies the formulation. From an optimization perspective, reducing the value of $C_u$ makes the unlabeled loss function less non-convex. In particular, when $C_u  = 0$, the objective function~\eqref{eq:s3vm_loss} reflects a conventional SVM training problem. Hence, the hyperparameter $C_u$ allows for the control of the overall objective's non-convexity. Optimizing~\eqref{eq:s3vm_loss} results in a linear decision function, representing a separating hyperplane. Nonlinear boundaries can be achieved through the kernel trick~\cite{cortes1995support, piccialli2018nonlinear}. The recently developed optimization techniques for minimizing the S3VM objective can be broadly categorized into two groups: combinatorial and continuous approaches. Combinatorial approaches treat the labels $\{y_i\}_{i=l+1}^n$ of unlabeled data points as explicit optimization variables. In contrast, continuous approaches do not consider unknown labels as optimization variables but instead minimize modified versions of~\eqref{eq:s3vm_loss} using continuous optimization techniques~\citep{chapelle2008optimization}. Our methodology bridges aspects from both worlds, focusing on a continuous optimization formulation while leveraging its underlying combinatorial nature.

\subsection{Problem formulation}\label{sec:prob_form}
Consider a binary classification problem with $n$ data points where the first $l$ data points are already labeled and the remaining data points $l + 1, \dots, n$ are unlabeled. We assume that all data points are centered around the origin so that the bias term $b$ can be set to zero~\cite{xu2004maximum}. To handle datasets that are not linearly separable, we map all data points into a higher-dimensional space using a mapping function $\phi: \mathbb{R}^d \rightarrow \mathbb{R}^m$ and then separate them using a hyperplane in this expanded space. This gives rise to the kernel-based S3VMs model, which can be formulated as the following MIP:
\begin{equation}\label{eq:S3VMmixed}
\begin{aligned}
\min_{w,\xi, y^u} \quad & \frac{1}{2} \|w\|_2^2 + C_l \sum_{i=1}^l\xi_i^2 + C_u \sum_{i=l+1}^n\xi_i^2 \\
\textrm{s.\,t.} \quad & y_i w^\top \phi(x_i) \geq 1 - \xi_i, \qquad i = 1, \dots, n,  \\
& y^u \coloneqq (y_{l+1},\dots, y_n)^\top \in \{-1,+1\}^{n - l}, \\
\end{aligned}
\end{equation}
where we use the L2-norm penalization for misclassification errors $\xi_i$.
For a given set of labels $y^u$, we consider the convex problem for building the SVM as
\begin{equation}\label{eq:S3VMprimal}
\begin{aligned}
\min_{w,\xi} \quad & \frac{1}{2} \|w\|_2^2 + C_l \sum_{i=1}^l\xi_i^2 + C_u \sum_{i=l+1}^n\xi_i^2 \\
\textrm{s.\,t.} \quad & y_i w^\top \phi(x_i) \geq 1 - \xi_i, \qquad i = 1, \dots, n. \\
\end{aligned}
\end{equation}
Since the mapping $\phi$ cannot be handled explicitly, a kernel function defined by $k(x_i, x_j) \coloneqq \phi(x_i)^\top \phi(x_j)$, which computes the inner product of the input vectors in a higher-dimensional space, is used instead. Let $\alpha \in \mathbb{R}^n$ be the vector of dual multipliers associated with the primal constraints. Moreover, let $\bar{K}$ be the kernel matrix, i.e., a symmetric {and positive semidefinite} matrix whose elements are computed by applying the kernel function to all pairs of input vectors, and let $D$ be a diagonal matrix with $D_{ii} = \frac{1}{2C_l}$ for $i=1, \dots, l$ and $D_{ii} = \frac{1}{2C_u}$ for $i=l+1, \dots, n$. Then the {Wolfe} dual~\cite{piccialli2018nonlinear} of Problem~\eqref{eq:S3VMprimal} is 
\begin{equation}\label{eq:S3VMdual}
\begin{aligned}
\max_{\alpha} \quad & e^\top \alpha-\frac{1}{2} \alpha^\top\left( K\circ yy^\top\right)\alpha \\
\textrm{s.\,t.} \quad & \alpha \geq 0,
\end{aligned}
\end{equation}
{where $K \coloneqq \bar{K} + D$ is a positive definite matrix.}

By looking at the Karush-Kuhn-Tucker (KKT) conditions and denoting by $\mu\in\mathbb{R}^n$ the multipliers of the nonnegativity constraints $\alpha\ge 0$, we derive the closed-form expression
\begin{equation}\label{eq:S3VMKKTdual}
 \alpha^\star=\left(K\circ yy^\top\right)^{-1}(e +\mu)
\end{equation}
for the solution of the dual. Recall that complementarity conditions imply that
\begin{equation}\label{eq:compl}
    \alpha_i\mu_i=0, \quad i=1,\dots,n.
\end{equation}
Combining~\eqref{eq:S3VMKKTdual} and~\eqref{eq:compl} and exploiting that $\left(K\circ yy^\top\right)^{-1} = K^{-1}\circ yy^\top$ and that for any $x, y\in \mathbb{R}^n, \ Q \succ 0$, it holds $x^\top(Q\circ yy^\top)x = (x\circ y)^\top Q(x\circ y)$, we can rewrite Problem~\eqref{eq:S3VMdual} as
\begin{equation*}
\begin{aligned}
\min_{\mu} \quad & \frac 1 2 \left((\mu+e)\circ y\right)^\top K^{-1}\left((\mu+e)\circ y\right)\\
\textrm{s.\,t.} \quad & \mu \geq 0.
\end{aligned}
\end{equation*}
We then replace Problem~\eqref{eq:S3VMprimal} in Problem~\eqref{eq:S3VMmixed} and end up with the following MIP:
\begin{equation}\label{eq:S3VMdualmu2}
\begin{aligned}
\min_{\mu, y^u} \quad & \frac 1 2 \left((\mu+e)\circ y\right)^\top K^{-1}\left((\mu+e)\circ y\right)\\
\textrm{s.\,t.} \quad & \mu \geq 0, \ y \coloneqq (y^l; y^u)^\top, ~ y^u \in \{-1, 1\}^{n-l}. \\
\end{aligned}
\end{equation}
Now, if we set $v = (\mu + e) \circ y$, the (convex) objective function becomes $\frac{1}{2}v^TK^{-1}v$. The constraints $\mu_i\ge 0$ become $y_iv_i\ge 1$ for labeled points and $v_i^2\ge 1$ for unlabeled points, yielding the following non-convex QCQP~\cite{bai2016conic}:
\begin{equation}\label{eq:orig_prob}
\begin{aligned}
\min_{v} \quad & \frac 1 2 v^\top K^{-1} v\\
\textrm{s.\,t.} \quad & y_i v_i \geq 1, \qquad i = 1,\dots,l, \\
& v_i^2 \geq 1, \qquad i = l + 1,\dots, n, \\
& v \in \mathbb{R}^n. \\
\end{aligned}
\end{equation}
\begin{remark}
\label{remark:support_vector}
Given the complementarity condition~\eqref{eq:compl}, we have that whenever a data point $x_i$ is a support vector, meaning $\alpha_i>0$, it must hold that $\mu_i=0$. Therefore, since $v_i=(\mu_i+1)y_i$, we get $v_i = y_i$. However, this is only a necessary condition for $x_i$ being a support vector since we may have $\alpha_i=\mu_i=0$.
\end{remark} 
\noindent\newline
From a feasible solution $v$ of Problem~\eqref{eq:orig_prob}, it is possible to obtain a feasible solution $y$ of Problem~\eqref{eq:S3VMdualmu2} by setting $y_i = \sign(v_i)$ for $i=1, \dots, n$. Moreover, the objective function values of both problems coincide. 

It is known that minimizing the S3VM objective does not necessarily align with the performance of the S3VM classifier in terms of accuracy. This discrepancy is primarily attributed to the presence of multiple local minima~\citep{chapelle2006branch, chapelle2008optimization}. However, even when Problem~\eqref{eq:orig_prob} is solved to global optimality, we empirically found that there is no guarantee that the resulting optimal S3VM classifier will exhibit high accuracy. Consequently, lower objective values do not consistently translate into improved classification performance. To address this issue, we introduce a balancing constraint to Problem~\eqref{eq:orig_prob}. This constraint is also used to avoid trivial solutions (with all unlabeled data points being classified as belonging to a single class), especially when the number of labeled examples is significantly smaller than that of unlabeled ones~\citep{chapelle2008optimization}. To enhance the robustness of S3VM models, Chapelle and Zien~\citep{chapelle2005semi} proposed modifications to the model in~\eqref{eq:S3VMmixed} by introducing a balancing constraint under the assumption of a linear kernel:
\begin{equation}\label{eq:or_balancing}
    \frac{1}{(n-l)}\sum_{i=l+1}^n \textrm{sign}(w^\top x_i) = \frac{1}{l}\sum_{i=1}^l y_i.
\end{equation}
This constraint ensures that the fraction of positive and negative assignments to the unlabeled data points also equals the fraction found in the labeled data points. {For us, the equivalent of constraint~\eqref{eq:or_balancing} is:
\begin{equation}\label{eq:or_balancing_sign}
    \frac{1}{(n-l)}\sum_{i=l+1}^n \textrm{sign}(v_i) = \frac{1}{l}\sum_{i=1}^l y_i.
\end{equation}
However, since this constraint cannot be handled explicitly, we decided to approximate it with the following new linear constraint:
}
\begin{equation}\label{eq:our_balancing}
        \frac{1}{(n-l)}\sum_{i=l+1}^n v_i = \frac{1}{l}\sum_{i=1}^l y_i.
\end{equation}
{The above constraint is an approximation of~\eqref{eq:or_balancing_sign}. Its foundation lies in the connection between Problems~\eqref{eq:S3VMmixed} and~\eqref{eq:orig_prob}, where the sign of $v_i$ represents the decision for the \mbox{$i$-th} data point. Thanks to Remark~\ref{remark:support_vector}, all support vectors satisfy $v_i=y_i$. Therefore, if all the points were support vectors, the two constraints~\eqref{eq:or_balancing_sign} and~\eqref{eq:our_balancing} would coincide. However, we have $v_i = y_i(\mu_i + 1)$ for points that are not support vectors, which may result in $|v_i| > 1$. In this scenario, constraint~\eqref{eq:our_balancing} encourages a distribution of the signs of the $v_i$ that mirrors the class distribution among the labeled data points.
}

While some approaches employing the S3VM formulation in~\eqref{eq:orig_prob} do not incorporate any class balancing constraint (see, e.g.,~\cite{bai2016conic, tian2017new}), our specific focus is on solving the QCQP problem outlined in~\eqref{eq:orig_prob} while incorporating the balancing constraint~\eqref{eq:our_balancing}. This adjustment is made to enhance the robustness and accuracy of the globally optimal solution, as reported in Section~\ref{sec:ml}. Thus, our S3VM model becomes
\begin{equation}\label{eq:orig_prob_balancing}
\begin{aligned}
\min_{v} \quad & \frac 1 2 v^\top K^{-1} v\\
\textrm{s.\,t.} \quad & y_i v_i \geq 1, \qquad i = 1,\dots,l, \\
& v_i^2 \geq 1, \qquad i = l + 1,\dots, n, \\
&  \frac{1}{(n-l)}\sum_{i=l+1}^n v_i = \frac{1}{l}\sum_{i=1}^l y_i, \\
& v \in \mathbb{R}^n. \\
\end{aligned}
\end{equation}
From a practical standpoint, non-convex QCQP problems are one of the most challenging optimization problems: the current size of instances that can be solved to provable optimality by general-purpose solvers remains very small. Therefore, our aim is twofold: (i) relaxing the intractable program in~\eqref{eq:orig_prob_balancing} into a tractable conic program, yielding computationally efficient bounds; and (ii) developing a branch-and-cut algorithm capable of solving large-scale instances to global optimality.

\subsection{Convex relaxations}\label{sec:conv_rel}
{In this section, we review the convex relaxations of Problem~\eqref{eq:orig_prob} proposed in the literature and introduce our starting relaxation. The literature explores two main alternatives: QP and SDP relaxations.}
\newline\newline
{\textbf{QP relaxations}} The standard QP relaxation is obtained by dropping the non-convex constraints $v_i^2 \ge 1$ for $i=l+1, \dots, n$ {in Problem~\eqref{eq:orig_prob}}. Thus, lower bounds can be computed by solving the resulting convex QP. Alternatively, the authors in~\cite{tian2017new} apply Lagrangian relaxation to Problem~\eqref{eq:orig_prob}, penalizing violations of non-convex constraints in the objective function. The Lagrange multipliers are estimated by solving an auxiliary SDP. More in detail, for a fixed $\lambda \geq 0$, Problem~\eqref{eq:orig_prob} can be relaxed to the following problem:
\begin{equation}\label{eq:qp_pen_bound}
\begin{aligned}
\min_{v} \quad & \frac{1}{2} v^\top (K^{-1} - 2\textrm{Diag}(\lambda)) v + e^\top \lambda \\
\textrm{s.\,t.} \quad & y_i v_i \geq 1, \qquad i = 1, \dots, l\\
& v \in \mathbb{R}^n.
\end{aligned}
\end{equation}
The corresponding vector of Lagrange multipliers $\lambda$ is estimated in~\cite{tian2017new} by solving the auxiliary SDP
\begin{equation}\label{eq:qp_pen_bound_aux}
\begin{aligned}
\max_{\lambda} \quad & e^\top \lambda \\
\textrm{s.\,t.} \quad & K^{-1} - 2\textrm{Diag}(\lambda) \succeq 0,\\
& \lambda \geq 0.
\end{aligned}
\end{equation}
\newline\newline
\noindent
{\textbf{SDP relaxations}} Non-convex QCQPs can also be
approximately solved via tractable SDP relaxations. Specifically, Problem~\eqref{eq:orig_prob} can be reformulated as the following non-convex SDP problem in a lifted space:
\begin{equation*}
\begin{aligned}
\min_{v, V} \quad & \frac{1}{2}\langle K^{-1} , V \rangle \\
\textrm{s.\,t.} \quad & y_iv_i \geq 1, \qquad i = 1,\dots,l, \\
& \operatorname{diag}(V) \geq e, \\
& V = vv^\top, \ v \in \mathbb{R}^n.\\
\end{aligned}
\end{equation*}
By relaxing the constraint $V=vv^\top$ into $V-vv^\top \succeq 0$ and applying the Schur complement, we end up with the basic SDP relaxation already considered in~\cite{bai2016conic}:
\begin{equation}\label{eq:SDP_basicrel}
\begin{aligned}
\min_{\bar{V}} \quad & \frac 1 2\langle K^{-1} , V \rangle \\
\textrm{s.\,t.} \quad & y_iv_i \geq 1, \qquad i = 1,\dots,l, \\
& \operatorname{diag}(V) \geq e, \\
& \bar{V} = \begin{pmatrix}
		V & v\\
		v^\top & 1\\
		\end{pmatrix} \succeq 0, ~ v \in \mathbb{R}^n, ~ V \in \mathbb{S}^n.
\end{aligned}
\end{equation}
Problem~\eqref{eq:SDP_basicrel} provides a lower bound on the optimal objective value of Problem~\eqref{eq:orig_prob}. Clearly, the bound provided by Problem~\eqref{eq:SDP_basicrel} is stronger than that provided by the QP relaxation. Tighter bounds can be obtained by solving the Doubly Nonnegative relaxation (DNN) of Problem~\eqref{eq:S3VMdualmu2}. To this end, Bai and Yan~\citep{bai2016conic} rewrite Problem~\eqref{eq:S3VMdualmu2} as a MIP using 0-1 variables. Then, they reformulate the problem as the following non-convex QCQP:
\begin{equation*}
\begin{aligned}
\min_{u} \quad & \frac 1 2 u^\top \tilde{K}^{-1} u\\
\textrm{s.\,t.} \quad & y_i(2u_i - u_{i+n}) \geq 1, \qquad i=1, \dots, l, \\
& u_j \geq 0, ~ u_{j+n} \geq 1, ~ u_j u_{j+n} - u_j^2 = 0, \qquad j = 1, \dots, n,\\
& u \in \mathbb{R}^{2n},
\end{aligned}
\end{equation*}
where $\tilde{K} \coloneqq (4K^{-1}, -2K^{-1}; -2K^{-1}, K^{-1}) \in \mathbb{S}^{2n}$. By relaxing the constraint $U = uu^\top$ into $U - uu^\top \succeq 0$ and applying the Schur complement, the DNN relaxation reads
\begin{equation}\label{eq:DNN_large}
\begin{aligned}
\min_{\bar{U}} \quad & \frac{1}{2} \langle \bar{K}^{-1} , U \rangle \\
\textrm{s.\,t.} \quad & y_i(2u_i - u_{i+n}) \geq 1, \qquad i=1, \dots, l, \\
& U_{j, j+n} - U_{j, j} = 0, \ u_{j+n} \geq 1, \qquad j=1, \dots, n, \\
& U_{k, k} - u_k \geq 0, \qquad k=1, \dots, 2n, \\
& \bar{U} = \begin{pmatrix}
		1 & u^\top\\
		u & U\\
		\end{pmatrix} \succeq 0, ~ \bar{U} \geq 0.
\end{aligned}
\end{equation}
Although Problem~\eqref{eq:DNN_large} is a convex, its size is larger than that of the basic SDP~\eqref{eq:SDP_basicrel}. Specifically, Problem~\eqref{eq:DNN_large} involves optimizing a matrix variable of dimensions \mbox{$(2n + 1) \times (2n + 1)$} over the intersection of the positive semidefinite cone and the nonnegative orthant. Moreover, according to our computational experiments, the DNN relaxation~\eqref{eq:DNN_large} provides a lower bound that is only slightly better than that of the basic SDP~\eqref{eq:SDP_basicrel}, and the computational effort to obtain such bound is much higher. Since existing SDP solvers struggle with large-size problems, the DNN relaxation~\eqref{eq:DNN_large} is not directly applicable within a branch-and-bound approach. This motivated us to strengthen the SDP relaxation~\eqref{eq:SDP_basicrel} instead, leading to the derivation of a new relaxation and valid constraints. These enhancements provide high-quality and computationally efficient bounds when embedded into a branch-and-cut algorithm.

\section{Branch-and-cut Approach }\label{sec:bac}
In this section, we present the main ingredients of our branch-and-cut algorithm. Specifically, we describe the lower bound computation, the primal heuristic, and the branching scheme. We start by rewriting Problem~\eqref{eq:orig_prob_balancing} in the more convenient form
\begin{equation}\label{eq:mod_prob}
\begin{aligned}
\min_{x} \quad & x^\top C x\\
\textrm{s.\,t.} \quad & 1 = L_i \leq x_i, \qquad i = 1,\dots,l \ : \ y_i=1,\\
 &  x_i\leq U_i=-1, \qquad i = 1,\dots,l \ : \ y_i=-1,\\
& x_i^2 \geq 1, \qquad i = 1,\dots, n, \\
&  \frac{1}{(n-l)}\sum_{i=l+1}^n x_i = \frac{1}{l}\sum_{i=1}^l y_i, \\
& x \in \mathbb{R}^n, \\
\end{aligned}
\end{equation}
where $C = \frac{1}{2} K^{-1}$. In Problem~\eqref{eq:mod_prob}, we have substituted the constraints $y_i x_i \geq 1$, $i = 1,\dots,l$, by writing them as explicit lower or upper bound constraints on the variables $x_i$. {From now on, we focus on defining our branch-and-cut approach for solving Problem~\eqref{eq:mod_prob} to global optimality.}

\subsection{{Lower bound computation}}\label{sec:lb_bt}

{
The success of any branch-and-cut algorithm relies on the quality and computational efficiency of its lower bounds. Strong lower bounds can significantly prune the search space, resulting in faster convergence and reduced computational effort. In the context of our problem, we utilize various global optimization tools to develop a sophisticated bounding routine, which is crucial for the overall efficiency of the algorithm. The first step is to introduce valid box constraints.}
\newline\newline
{\textbf{Bound tightening}}
The feasible set of Problem~\eqref{eq:mod_prob} is unbounded. Nevertheless, it is possible to introduce optimality-based box constraints $L_i \leq x_i \leq U_i$ for all variables $x_i$ once we have an upper bound on the optimal objective value of Problem~\eqref{eq:orig_prob_balancing}. This is especially desirable for variables corresponding to unlabeled data points to derive further valid constraints later. The idea is to solve auxiliary convex optimization problems to compute lower and upper bounds on all the variables such that at least one globally optimal solution of Problem~\eqref{eq:orig_prob_balancing} remains within the restricted feasible set.

Given a non-convex optimization problem and a tractable convex relaxation of it, classical optimality-based bound tightening computes the tightest bounds valid for all relaxation solutions by in turn minimizing and maximizing each variable, and imposing an objective cutoff~\cite{gleixner2017three}. If $UB$ denotes an upper bound for Problem~\eqref{eq:orig_prob_balancing}, then we can compute a new lower bound ${L}_i$ and a new upper bound ${U}_i$ for variable $x_i$ by solving QCQPs of the form
\begin{equation}\label{eq:boundxi}
\begin{aligned}
{L}_i / {U}_i = \min / \max \quad & x_i \\
\textrm{s.\,t.} \quad & L_i \leq x_i \leq U_i, \qquad i = 1,\dots,n, \\
		& x^\top C x \leq UB, \\
            &  \frac{1}{(n-l)}\sum_{i=l+1}^n x_i = \frac{1}{l}\sum_{i=1}^l y_i, \\
		& x \in \mathbb{R}^n, \\
\end{aligned}
\end{equation}
where some bound constraints are $- \infty$ or $+ \infty$ in the beginning. {With abuse of notation, we refer to $L_i$ and $U_i$ both as the initial bounds on the variables $x_i$ in the feasible set definition and the updated bounds obtained from solving problems of type~\eqref{eq:boundxi}, with the new values replacing the old ones whenever a bound constraint is updated.} Since $C$ is a positive definite matrix, Problem~\eqref{eq:boundxi} is a convex and can be solved efficiently. Applying a complete round of optimality-based bound tightening involves solving up to $2n$ convex QCQPs, each with one quadratic constraint. Note that the lower bound computation for variable $x_i$ can be skipped if $L_i \geq 1$ already holds since the optimal objective value of~\eqref{eq:boundxi} would most likely be $L_i$ in this case. The same applies to the upper bound computation if $U_i \leq -1$ already holds.
\\
\begin{remark}[Projecting box constraints]\label{rem2}
Whenever the optimal objective value of Problem~\eqref{eq:boundxi} yields a new lower bound $L_i>-1$, we can set $L_i$ to $L_i \coloneqq \max\{L_i, 1\}$. This is due to the non-convex constraints $x_j^2 \geq 1$ in Problem~\eqref{eq:orig_prob_balancing} that require the absolute value of each variable to be at least one. Similarly, we can set $U_i$ to $U_i \coloneqq \min\{U_i,-1\}$ if the optimal objective value of Problem~\eqref{eq:boundxi} yields an upper bound $U_i < 1$. In both cases, the label of the corresponding data point ($-1$ or $+1$) becomes fixed, and the non-convex constraint $x_i^2 \geq 1$ could be removed from Problem~\eqref{eq:orig_prob_balancing}. This reduces the size of the branch-and-bound tree and significantly impacts the overall efficiency of the algorithm.
\end{remark}
\noindent\newline
Remark~\ref{rem2} gives a strong motivation to recompute the box constraints from time to time if new information is available. In fact, we solve all QCQPs of type~\eqref{eq:boundxi} again whenever we have updated the best known upper bound $UB$. 
\newline\newline
{\textbf{Basic SDP relaxation}}
Once lower and upper bounds on each variable have been computed by solving problems of type~\eqref{eq:boundxi}, Problem~\eqref{eq:mod_prob} can be equivalently stated as the non-convex matrix optimization problem
\begin{equation}\label{eq:matrix_prob}
\begin{aligned}
\min_{x, X} \quad & \langle C , X \rangle \\
\textrm{s.\,t.} \quad & L_i \leq x_i \leq U_i, \qquad i = 1,\dots,n, \\
& 1 \leq X_{ii} \leq \max\{L_i^2,U_i^2\}, \qquad i = 1,\dots,n \\
&  \frac{1}{(n-l)}\sum_{i=l+1}^n x_i = \frac{1}{l}\sum_{i=1}^l y_i, \\
& X = xx^\top, \ x \in \mathbb{R}^n.\\
\end{aligned}
\end{equation}
In Problem~\eqref{eq:matrix_prob}, we have also bounded the main diagonal of $X$ via the constraints $X_{ii} \leq \max\{L_i^2,U_i^2\}$ for $i = 1,\dots,n$. These can be derived by using the box constraints on $x_i$ and the fact that $x_i^2 = X_{ii}$ must hold in any feasible solution of~\eqref{eq:matrix_prob}.

By relaxing the constraint $X=xx^\top$ into $X-xx^\top\succeq 0$ and applying the Schur complement, we end up with our first SDP relaxation
\begin{equation}\label{eq:our_basic}
\begin{aligned}
\min_{\bar{X}} \quad & \langle C , X \rangle \\
\textrm{s.\,t.} \quad & L_i \leq x_i \leq U_i, \qquad i = 1,\dots,n, \\
& 1 \leq X_{ii} \leq \max\{ L_i^2, U_i^2 \}, \qquad i = 1,\dots,n, \\
&  \frac{1}{(n-l)}\sum_{i=l+1}^n x_i = \frac{1}{l}\sum_{i=1}^l y_i, \\
& \bar{X} = \begin{pmatrix}
		X & x\\
		x^\top & 1\\
		\end{pmatrix} \succeq 0, ~ x \in \mathbb{R}^n, ~ X \in \mathbb{S}^n.
\end{aligned}
\end{equation}
\begin{proposition}
Assume that $L_i \in \mathbb{R}$ and $U_i \in \mathbb{R}$ are computed by solving problems of type~\eqref{eq:boundxi} for $i=1,\ldots,n$. Then, Problem~\eqref{eq:our_basic} provides a lower bound at least as strong as the one provided by Problem~\eqref{eq:SDP_basicrel}.
\end{proposition}
\begin{proof}
    The assumption on $L_i$ and $U_i$ implies that the box constraints $L_i\le x_i\le U_i$ are valid for Problem~\eqref{eq:mod_prob} for $i=1, \dots, n$. Therefore, Problem~\eqref{eq:our_basic} is a relaxation of Problem~\eqref{eq:orig_prob_balancing}. The result then follows from the feasible set of Problem~\eqref{eq:our_basic} being strictly contained in the feasible set of Problem~\eqref{eq:SDP_basicrel}.
\end{proof}
\noindent
{\textbf{Valid inequalities}} The SDP relaxation~\eqref{eq:our_basic} can be strengthened by adding constraints that are also valid for any feasible solution of Problem~\eqref{eq:matrix_prob}. Since we have box constraints, the constraints coming from the Reformulation-Linearization Technique (RLT)~\cite{sherali1995reformulation} are an interesting class for us. They can be derived by multiplying the nonnegative expressions $x_i - L_i \geq 0$ and $U_i - x_i \geq 0$ by the nonnegative expressions $x_j - L_j \geq 0$ and $U_j - x_j \geq 0$ for all $i,j = 1,\dots,n, \ i < j$, and then expanding to
\begin{equation}\label{RLT}
	\begin{split}
		X_{ij} &\geq \max \{ U_i x_j + U_j x_i - U_i U_j , \ L_i x_j + L_j x_i - L_i L_j \}, \\
		X_{ij} &\leq \min \{ L_i x_j + U_j x_i - L_i U_j , \ U_i x_j + L_j x_i - U_i L_j \}. \\
	\end{split}
\end{equation}
It will become evident in the computational results in Section~\ref{sec:results} that adding RLT inequalities~\eqref{RLT} significantly improves the SDP relaxation~\eqref{eq:our_basic}. In the end, we obtain lower bounds for Problem~\eqref{eq:orig_prob_balancing} by solving the SDP
\begin{equation}\label{eq:SDP_rel_RLT}
\begin{aligned}
\min_{\bar{X}} \quad & \langle C , X \rangle \\
\textrm{s.\,t.} \quad & L_i \leq x_i \leq U_i, \qquad i = 1,\dots,n, \\
& 1 \leq X_{ii} \leq \max\{ L_i^2, U_i^2 \}, \qquad i=1,\dots, n,\\
& \frac{1}{(n-l)}\sum_{i=l+1}^n x_i = \frac{1}{l}\sum_{i=1}^l y_i, \\
& X_{ij} \geq U_i x_j + U_j x_i - U_i U_j, \qquad i,j = 1,\dots,n,~ i<j, \\
		& X_{ij} \geq L_i x_j + L_j x_i - L_i L_j, \qquad i,j = 1,\dots,n,~ i<j, \\
		& X_{ij} \leq L_i x_j + U_j x_i - L_i U_j, \qquad i,j = 1,\dots,n,~ i<j, \\
		& X_{ij} \leq U_i x_j + L_j x_i - U_i L_j, \qquad i,j = 1,\dots,n,~ i<j. \\
& \bar{X} = \begin{pmatrix}
		X & x\\
		x^\top & 1\\
		\end{pmatrix} \succeq 0, \ x \in \mathbb{R}^n, \ X \in \mathbb{S}^n.
\end{aligned}
\end{equation}
Despite having only $\mathcal{O}(n^2)$ RLT inequalities, adding them all at once would make the relaxation intractable, even for a moderate number of data points, when using off-the-shelf interior-point solvers. Thus, we treat them as cutting planes and only add a few of the most violated ones during each cutting-plane iteration.

Further valid constraints for~\eqref{eq:matrix_prob} can be generated by multiplying the balancing constraint~\eqref{eq:our_balancing} by any variable $x_j, \ j = 1,\dots,n$, resulting in the constraints
\begin{equation*}
    \frac{1}{(n-l)}\sum_{i=l+1}^n x_i x_j = \left(\frac{1}{l}\sum_{i=1}^l y_i\right)x_j, \qquad j = 1,\dots,n.
\end{equation*}
These product constraints can then be linearized and added to the SDP relaxation~\eqref{eq:SDP_rel_RLT} by replacing any quadratic term $x_ix_j$ by $X_{ij}$. Adding these constraints improves the lower bound but slows down computation. Therefore, they can be beneficial for problems with a weak lower bound, but otherwise, they can be neglected.


For several reasons, the box constraints $L_i \leq x_i \leq U_i$ play a crucial role in our setting. Firstly, they indicate whether the sign of a variable $x_i$ is known, which is the case whenever $L_i \geq 1$ or $U_i \leq -1$ holds. This reduces the degree of non-convexity in the problem since the constraint $x_i^2 \geq 1$ can be removed. Secondly, tighter box constraints can significantly impact the quality of our lower bound as the RLT inequalities become more restrictive. The box constraints can be improved by resolving problems of type~\eqref{eq:boundxi} whenever a better upper bound for~\eqref{eq:orig_prob_balancing} is found. Additionally, every branching step also impacts the box constraints, see Section~\ref{sec:branching}.
\newline\newline
{\textbf{Marginals-based bound tightening}}
Moreover, to further tighten the box constraints, we use a standard technique in global optimization, namely marginals-based bound tightening as proposed in~\cite{ryoo1996branch}. The general idea is the following, see~\cite[Theorem~2]{ryoo1996branch}: let $g(x,X)\leq 0$ be an active constraint in our SDP relaxation~\eqref{eq:SDP_rel_RLT} with corresponding dual multiplier $\lambda >0$ at the optimal solution with objective value $LB$. Moreover, let $UB$ be an upper bound on Problem~\eqref{eq:mod_prob}. Then the constraint
\begin{equation}\label{marginal}
    g(x,X) \geq - \frac{UB - LB}{\lambda}
\end{equation}
is valid for all solutions of Problem~\eqref{eq:mod_prob} with objective value better than $UB$. In other words, the constraint~\eqref{marginal} may cut off feasible points but preserves at least one globally optimal solution.

An important special case of this idea arises when applying it to the box constraints $L_i \leq x_i \leq U_i$: if $L_i \leq x_i$ is active at the optimum of~\eqref{eq:SDP_rel_RLT} with dual multiplier $\lambda_i^L > 0$, then we can update the upper bound $U_i$ via
\begin{equation*}
    U_i \coloneqq \min \left\{ U_i , L_i + \frac{UB - LB}{\lambda_i^L}\right\}.
\end{equation*}
Analogously, if $x_i \leq U_i$ is active with dual multiplier $\lambda_i^U > 0$, then we can update the lower bound $L_i$ via
\begin{equation*}
L_i \coloneqq \max \left\{ L_i , U_i - \frac{UB - LB}{\lambda_i^U} \right\}.
\end{equation*}
Additionally, we can exploit the non-convex constraint $X = xx^\top$ in Problem~\eqref{eq:mod_prob} by applying the same idea to the main diagonal of $X$. \\
\begin{lemma}
Let $i \in \{1,\dots,n\}$. If the constraint $X_{ii} \geq 1$ is active at the optimal solution of~\eqref{eq:SDP_rel_RLT} with dual multiplier $\lambda > 0$, then we can update $L_i$ and $U_i$ via
\begin{equation*}
L_i \coloneqq \max \left\{ L_i , - \sqrt{1 + \frac{UB - LB}{\lambda}} \right\},\quad U_i \coloneqq \min \left\{ U_i , \sqrt{1 + \frac{UB - LB}{\lambda}} \right\}.
\end{equation*}
\end{lemma}
\begin{proof}
    We rewrite the constraint $X_{ii} \geq 1$ as $-X_{ii} + 1 \leq 0$. Using~\eqref{marginal}, the constraint 
    \begin{equation*}
    - X_{ii} + 1 \geq - \frac{UB - LB}{\lambda}
    \quad\Leftrightarrow\quad
    X_{ii} \leq 1 + \frac{UB - LB}{\lambda} \eqqcolon \delta
    \end{equation*}
    preserves at least one globally optimal solution and can be added to the SDP relaxation~\eqref{eq:SDP_rel_RLT} and the non-convex Problem~\eqref{eq:matrix_prob}. Moreover, $x_i^2 = X_{ii}$ must hold in any feasible solution of Problem~\eqref{eq:matrix_prob}. Thus, the constraints $-\sqrt{\delta} \leq x_i \leq \sqrt{\delta}$ also preserve at least one globally optimal solution.
\end{proof}
Furthermore, our SDP relaxation~\eqref{eq:SDP_rel_RLT} includes constraints of the form $X_{ii} \leq \max\{L_i^2,U_i^2\}$, which facilitate further tightening of the box constraints. \\
\begin{lemma}\label{lem:tightening}
    Let $i \in \{1,\dots,n\}$. Assume that a constraint of type $X_{ii} \leq \gamma$ is active at the optimal solution of~\eqref{eq:SDP_rel_RLT} with dual multiplier $\lambda > 0$ such that $p \coloneqq \gamma - \frac{UB - LB}{\lambda} \geq 1$. Then the following holds:
    \begin{enumerate}[label=(\roman*)]
        \item If $L_i > - \sqrt{p}$, then we can update $L_i$ via $L_i \coloneqq \max\{L_i,\sqrt{p}\}$.
        \item If $U_i < \sqrt{p}$, then we can update $U_i$ via $U_i \coloneqq \min\{U_i,-\sqrt{p}\}$.
    \end{enumerate}
\end{lemma}
\begin{proof}
    Applying~\eqref{marginal} to the constraint $X_{ii} - \gamma \leq 0$ yields the new constraint
    \begin{equation*}
    X_{ii} - \gamma \geq - \frac{UB - LB}{\lambda}
    \quad\Leftrightarrow\quad
    X_{ii} \geq p \geq 1.
    \end{equation*}
    We know that $x_i^2 = X_{ii}$ must hold in any feasible solution of Problem~\eqref{eq:matrix_prob}. If $x_i^2 \geq p$ and $-\sqrt{p} < L_i \leq x$, then certainly $x_i \geq \sqrt{p}$ must hold. Analogously, if $x_i^2 \geq p$ and $x_i \leq U_i < \sqrt{p}$, then $x_i  \leq -\sqrt{p}$ must hold.
\end{proof}
Lemma~\ref{lem:tightening} establishes another scenario in which the degree of non-convexity can be reduced by fixing the sign of a variable. The main advantage of marginals-based bound tightening is that it is computationally inexpensive to apply. Thus, we apply it in every cutting-plane iteration after solving the SDP relaxation~\eqref{eq:SDP_rel_RLT} and before separating violated RLT inequalities~\eqref{RLT}.

\noindent
{Overall, our bounding procedure involves four steps:
\begin{enumerate}
    \item Box constraints computation by solving problems~\eqref{eq:boundxi} whenever the global upper bound has been updated.
    \item Solution of the basic SDP relaxation~\eqref{eq:our_basic}.
    \item Cutting-plane algorithm for adding violated RLT inequalities~\eqref{RLT} and solution of the corresponding SDP problem~\eqref{eq:SDP_rel_RLT}.
    \item Improvement of box constraints using marginals-based bound tightening.
\end{enumerate}
}

During the cutting-plane approach, we ensure that the sizes of SDPs remain manageable by removing (almost) inactive RLT cuts at each iteration. Moreover, we impose a minimum violation of any cut to be added to the relaxation and restrict the number of new cuts per iteration. We use the SDP solution at each iteration to apply the primal heuristic described in Section~\ref{sec:heuristic}. The cutting-plane approach is stopped if the relative change of the lower bound compared to the last iteration is smaller than a fixed threshold. If necessary, we then apply the branching strategy presented in Section~\ref{sec:branching}. In Section~\ref{sec:bound_comp}, we report a numerical comparison of different lower bounding approaches that confirms the effectiveness of our approach.

\subsection{Primal heuristic}\label{sec:heuristic}
An essential component of any branch-and-bound methodology is the incorporation of primal heuristics to find good primal feasible solutions. In our case, any feasible solution of~\eqref{eq:orig_prob_balancing} yields an upper bound. Since Problem~\eqref{eq:orig_prob_balancing} involves continuous variables only, finding high-quality upper bounds can be difficult and time-consuming. However, producing nearly optimal upper bounds early in the branch-and-bound approach is particularly crucial for the overall efficiency of our approach. Firstly, solving SDPs is a relatively expensive task to do. Exploring fewer branch-and-bound nodes is always beneficial and leads to a noticeable speedup. Secondly, stronger upper bounds also allow us to compute tighter box constraints on each variable by (re-)solving the convex QCQPs~\eqref{eq:boundxi}. This again improves the RLT inequalities~\eqref{RLT} and eventually leads to stronger lower bounds, reducing the optimality gap. Thirdly, marginals-based bound tightening also uses the best-known upper bound, yielding stronger box constraints if better upper bounds are known.

Let $(\hat x,\hat X)$ be the optimal solution of~\eqref{eq:SDP_rel_RLT}. Unless $(\hat x,\hat X)$ forms a rank-one solution, i.e., $\hat X = \hat x \hat x^\top$ holds, $\hat x$ is not feasible for the non-convex problem~\eqref{eq:orig_prob_balancing}. Even if we neglect the balancing constraint~\eqref{eq:our_balancing}, we would most likely get a solution of poor quality by just projecting each entry $\hat x_i$ of $\hat x$ onto the feasible region.

However, it is possible to improve any solution $\bar x \in \mathbb{R}^n$ that is feasible for~\eqref{eq:orig_prob_balancing} by constructing a labeling vector $\bar y \in \{-1,+1\}^n$ with entries $\bar y_i = \operatorname{sign}(\bar x_i)$ for $i = 1,\dots,n$, and then solving the convex QP
\begin{equation}\label{eq:convex_qp}
\begin{aligned}
\min_{x} \quad & x^\top C x\\
\textrm{s.\,t.} \quad & \bar y_i x_i \geq 1, \qquad i = 1,\dots,n, \\
&  \frac{1}{(n-l)}\sum_{i=l+1}^n x_i = \frac{1}{l}\sum_{i=1}^l y_i, \\
& x \in \mathbb{R}^n. \\
\end{aligned}
\end{equation}
The optimal solution of~\eqref{eq:convex_qp} will be feasible for~\eqref{eq:orig_prob_balancing}, and it is the best solution one can get with respect to the fixed labeling vector $\bar y$. Solving the QP~\eqref{eq:convex_qp} is a non-trivial task within a heuristic. Still, it allows us to exploit the combinatorial structure of the problem: we now only have to find suitable labeling vectors with entries in $\{-1,+1\}$, which is much easier than directly finding feasible solutions with continuous entries.

Like suggested in~\cite{bai2016conic}, we exploit the optimal SDP solution $(\hat x,\hat X)$ of~\eqref{eq:SDP_rel_RLT} and first construct a labeling vector $\bar y$ with entries $\bar y_i = \operatorname{sign}(\hat x_i$) for $i=1,\dots,n$. We then solve the convex QP~\eqref{eq:convex_qp} with this labeling vector and obtain an upper bound on the non-convex problem~\eqref{eq:orig_prob_balancing}. However, it turns out that a solution found this way can often be improved. To do so, we use a two-opt local search approach, which we present below.

Suppose that $\bar x \in \mathbb{R}^n$ is feasible for~\eqref{eq:orig_prob_balancing}. The idea is to improve $\bar x$ by selecting two variables $x_i$ and $x_j$ with $i,j \in \{l+1,\dots,n\},\ i < j$, and then to solve Problem~\eqref{eq:orig_prob_balancing} with respect to $x_i$ and $x_j$ only, while keeping all other variables fixed at their values in $\bar x$. If doing so, Problem~\eqref{eq:orig_prob_balancing} reduces to
\begin{equation}\label{eq:two_opt}
\begin{aligned}
\min_{x_i, x_j} \quad & 
\begin{pmatrix}
x_i \\
x_j \\
\end{pmatrix}^\top
\begin{pmatrix}
C_{ii} & C_{ij} \\
C_{ij} & C_{jj} \\
\end{pmatrix}
\begin{pmatrix}
x_i \\
x_j \\
\end{pmatrix}
+ 2 \begin{pmatrix}
\eta_i \\
\eta_j \\
\end{pmatrix}^\top
\begin{pmatrix}
x_i \\
x_j \\
\end{pmatrix} \\
\textrm{s.\,t.} \quad & x_i^2 \geq 1, \\
& x_j^2 \geq 1, \\
& x_i + x_j = k, \\
& x_i, x_j \in \mathbb R, \\
\end{aligned}
\end{equation}
where $\eta_i = \sum_{k \neq i,j}^n C_{ki} \bar x_k$, $\eta_j = \sum_{k \neq i,j}^n C_{kj} \bar x_k$, and $k = \bar x_i + \bar x_j$. Problem~\eqref{eq:two_opt} still is a non-convex optimization problem, but we can solve it efficiently, as shown in the next proposition.\\
\begin{proposition}
    Problem~\eqref{eq:two_opt} admits a global minimum, and this minimum can be computed analytically.
\end{proposition}
\begin{proof}
We can eliminate the variable $x_i$ by substituting $x_i = k - x_j$. This gives us a univariate optimization problem of the form
\begin{equation}\label{univariate}
\begin{aligned}
\min_{x_j} \quad & ax_j^2 + bx_j + c \\
\textrm{s.\,t.} \quad & g_1(x_j)=x_j^2 \geq 1, \\
		& g_2(x_j)=(k-x_j)^2 \geq 1, \\
		& x_j \in \mathbb{R}, \\
\end{aligned}
\end{equation}
where $a = C_{ii} + C_{jj} - 2 C_{ij}>0, \ b = 2k (C_{ij} - C_{ii}) - 2 \eta_i + 2 \eta_j$, and $c = C_{ii} + 2k g_i$. The objective function is quadratic and positive definite. Therefore, it is coercive and admits a global minimum. Thus, the global minimum will be a point at which the minimum objective value among all candidates derived by necessary optimality conditions is attained. The Lagrangian of~\eqref{univariate} is given by
\begin{equation}
    \mathcal{L} (x_j;\lambda_1, \lambda_2) = ax_j^2 + bx_j + c - \lambda_1 (x_j^2 - 1) - \lambda_2 ((k-x_j)^2 - 1),
\end{equation}
where $\lambda_1$ and $\lambda_2$ are Lagrange multipliers corresponding to the inequality constraints. The KKT conditions are
\begin{align}
        2ax_j + b - 2\lambda_1 x_j + 2 \lambda_2 (k - x_j) &= 0, \label{dldx} \\
        \lambda_1 (x_j^2 - 1) &= 0, \label{compl1} \\
        \lambda_2 ((k - x_j)^2 - 1) &= 0, \label{compl2} \\
        x_j^2 &\geq 1, \label{constr1} \\
        (k-x_j)^2 &\geq 1, \label{constr2} \\
        \lambda_1,\lambda_2 &\geq 0. 
\end{align}
Given the two constraints, there are four cases:
\begin{enumerate}
    \item If $\lambda_1 = \lambda_2 = 0$, then no constraint is active. Then we get $x_j = -\frac{b}{2a}$ due to~\eqref{dldx}. Since this is the unconstrained global minimum, this solution is optimal if it is feasible, i.e., it satisfies~\eqref{constr1} and~\eqref{constr2}. 
    \item Suppose $\lambda_1 > 0$ and $\lambda_2 = 0$. We get $x_j \in \{-1,+1\}$ due to~\eqref{compl1}. Then the LICQ constraint qualification holds since $\frac{\partial g_1}{\partial x_j} = 2x_j\not=0 $. Moreover, solving~\eqref{dldx} for $\lambda_1$ yields $\lambda_1 = a + \frac{b}{2x_j} = a + \frac{b}{2}x_j$. For both candidates $x_j=-1$ and $x_j=+1$, we can check whether they are KKT points, i.e., $\lambda_1$ is nonnegative and~\eqref{constr2} holds.
    \item Suppose $\lambda_1 = 0$ and $\lambda_2 > 0$. We get $x_j \in \{k-1,k+1\}$ due to~\eqref{compl2}. Again, the LICQ constraint qualification holds since  $\frac{\partial g_2}{\partial x_j} = -2(k-x_j)\not=0 $. Solving~\eqref{dldx} for $\lambda_2$ yields $\lambda_2 = \frac{2ax_j + b}{2(k-x_j)}$. For both candidates, $x_j = k - 1$ and $x_j = k + 1$, we can check whether they are KKT points, i.e., $\lambda_2$ is nonnegative and~\eqref{constr1} holds.
    \item Suppose $\lambda_1 > 0$ and $\lambda_2 > 0$. We get $x_j \in \{-1,+1\}$ from~\eqref{compl1} and $x_j \in \{k-1,k+1\}$ from~\eqref{compl2}. This can only be true if $k \in \{-2,0,+2\}$. In this case, the point satisfies the Fritz John optimality conditions. Hence, it is a candidate for optimality.
\end{enumerate}
The global minimum of Problem~\eqref{eq:two_opt} is the candidate that achieves the minimum objective value.
\end{proof}
We apply the two-opt local search in the following way. We start with a feasible solution $\bar x \in \mathbb R^n$ of Problem~\eqref{eq:orig_prob_balancing}, and then loop through all possible pairs $\{i,j\}$ of unlabeled data points. For any pair, if the optimal solution of Problem~\eqref{eq:two_opt} is better than the solution $(\bar x_i,\bar x_j)$, we immediately update $\bar x$ and proceed with this updated solution. If all pairs have been tested and no improvement was found, we are done. Otherwise, we first update $\bar x$ by solving the convex QP~\eqref{eq:convex_qp} using the labeling vector $\bar y$ with entries $\bar y_i = \operatorname{sign}(\bar x_i)$ for $i = 1,\dots,n$, and then we check all pairs $\{i,j\}$ for any improvement again. {The pseudocode of this two-opt local search heuristic is given in Algorithm~\ref{alg:twoopt}.}

\begin{algorithm}[!ht]
\caption{Two-opt local search heuristic}
\label{alg:twoopt}
{
\textbf{Input}: Feasible solution $\bar x \in \mathbb R^n$ of Problem~\eqref{eq:orig_prob_balancing}. 
\begin{enumerate}[label*=\arabic*., nolistsep] 
\item For every pair $\{i,j\}$ of unlabeled data points:
\begin{enumerate}[label*=\arabic*., rightmargin=15pt, nolistsep]
    \item Compute an optimal solution $x_i,x_j$ of Problem~\eqref{eq:two_opt} with respect to $\bar{x}$ and the pair $\{i,j\}$.
    \item If the solution $x_i,x_j$ differs from $\bar{x}_i,\bar{x}_j$, then update $\bar{x}_i$ to $x_i$ and $\bar{x}_j$ to $x_j$.
\end{enumerate}
\item If $\bar{x}$ was updated in the last execution of \textit{Step 1}, then compute a new vector $\bar{x}$ by solving Problem~\eqref{eq:convex_qp} with respect to the labeling vector $\bar{y}$ with entries $\bar{y}_i = \operatorname{sign}(\bar{x}_i)$ for $i = 1, \dots, n$, and go to \textit{Step 1}.
\end{enumerate}
\textbf{Output}: Labeling vector $\bar{y}$ and upper bound $UB = \bar{x}^\top C \bar{x}$. 
}
\end{algorithm}

\subsection{Branching strategy}\label{sec:branching}
Problem~\eqref{eq:orig_prob_balancing} is a continuous optimization problem, but it is possible to enumerate all the solutions. Note that the Problem~\eqref{eq:orig_prob_balancing} becomes convex when all unlabeled data points have been labeled. Therefore, we select an unlabeled data point and derive two children corresponding to the two possible labels. In practice, labeling a point is equivalent to setting either the lower bound to $L_i=1$ ($y_i=1$) or the upper bound to $U_i=-1$ ($y_i=-1$). At a given node, some unlabeled points have already been assigned a label. This yields a binary enumeration tree where nodes are associated with a fixed partial labeling of the unlabeled dataset, and the two children correspond to the labeling of some new unlabeled point. 

The key question is how to select, at any branching step, a data point to label. We explore different options, and all of them exploit the solution of our SDP relaxation~\eqref{eq:SDP_rel_RLT} at the current node. To select interesting candidates, we do some preliminary considerations. Given the solution $(\bar x, \bar X)$ of the SDP relaxation at the end of the cutting-plane algorithm, we consider the labeling $\bar y_i=\sign(\bar x_i)$ for $i=l+1,\ldots,n$ and solve the convex QP~\eqref{eq:convex_qp}. Let $x^\star$ be the optimal solution of this problem.

We focus on the indices of unlabeled data points that satisfy $x^\star_i=1$ since they correspond to data points near the boundary of the decision function and, therefore, are more relevant. Moreover, we consider data points where the label assigned by the SDP solution appears to be undecided, i.e., $\bar x_i\in(-1,1)$, hoping to achieve a bound improvement in both child nodes. Indeed, we restrict the candidates for branching to the set
\begin{equation*}
    {\cal U} = \{i=l+1,\ldots,n \, \colon x^\star_i=1 \textrm{ and } \left\vert \bar x_i \right\vert < 1 \}.
\end{equation*}
To select the branching variable, we evaluate different measures that could help identify the variable leading to the largest bound improvement. We consider the approximation error, a quality measure that can be computed in various ways. In particular, we know that at the optimum $X_{ij}=x_ix_j$ must hold for all $i, j=1, \dots, n$. The goal then is to label the point $i$ that maximizes the approximation error, assuming that a larger approximation error indicates room for improvement on that variable. We consider four measures of the approximation error for variable $i\in {\cal U}$:
 \begin{align*}
     a_i^1 &= \sum_{j=1}^n \left(\bar x_i \bar x_j - \bar{X}_{ij}\right), \qquad  a_i^2 = \sum_{j=1}^n \left\vert \bar x_i \bar x_j - \bar X_{ij} \right\vert, \\
     a_i^3 &= \sum_{j=1}^n C_{ij}\left(\bar x_i \bar x_j - \bar X_{ij}\right), \qquad a_i^4 = \sum_{j=1}^n \left\vert C_{ij}\left(\bar x_i \bar x_j - \bar X_{ij} \right) \right\vert.
 \end{align*}
Another interesting measure of \enquote{undecidedness} is the box size around the variable $x_i$. Given that the point is unlabeled, we know that $L_i \leq -1$ and $U_i \geq 1$ must hold. If we select the data point $i$ for branching, then we have:
\begin{itemize}
\item In the child where we choose the label $y_i=1$, we set $L_i=1$, thereby reducing the box by the positive amount $1-L_i$.
\item In the child where we choose the label $y_i=-1$, we set $U_i=-1$, reducing the box by the positive amount $1+U_i$.
\end{itemize}
The intuition here is that the bound improvement is correlated with the reduction of the feasible set in the two children. Consequently, another interesting measure (the larger, the better) is defined as:
\begin{equation*}
b_i=\min_{i \in \mathcal{U}}\{1-L_i,1+U_i\}.
\end{equation*}
We then rank all data points in ${\cal U}$ in decreasing order with respect to $a_i^1$, $a_i^2$, $a_i^3$, $a_i^4$, and $b_i$, summing up all the positions in the ranking. The branching variable is selected as the one with the lowest cumulative score. 

The pseudocode of the overall branch-and-cut algorithm is given in Algorithm~\ref{alg:bbpseudocode} and uses all our ingredients. 
{
Steps 5.1--5.6 in Algorithm~\ref{alg:bbpseudocode} are executed sequentially for each node of the B\&B tree. Consequently, the complexity per node is dominated by solving the SDP relaxation, which has a complexity of $\mathcal{O}(n^{3.5})$ and is performed as many times as the number of cutting-plane iterations. Whenever the box constraints are updated, up to $2n$ convex QCQPs are solved using interior-point methods, each with a complexity of $\mathcal{O}(n^3)$.}

\begin{algorithm}[!ht]
\caption{Branch-and-cut algorithm for S3VM}
\label{alg:bbpseudocode}
\textbf{Input}: Kernel matrix $\bar K$, penalty parameters $C_l>0$ and $C_u>0$, labels $\{y_i\}_{i=1}^l$, upper bound $UB$, optimality tolerance $\varepsilon \geq 0$. 

\begin{enumerate}[label*=\arabic*., nolistsep] 
\item Set $L_i=1,~ U_i=\infty$ if $y_i=1$, $L_i=-\infty,~ U_i=-1$ if $y_i=-1$ for $i=1,\ldots,l$. Set $L_i=-\infty,~U_i=\infty$ for $i=l+1,\ldots,n$.
\item Compute valid box constraints $L_i$ and $U_i$ for $i=1,\ldots,n$.
    \item Let $P_0$ be the initial S3VM problem and set $\mathcal{Q} = \{P_0\}$. 
    \item Let $(y^u)^\star$ be a labeling with objective function value $v^\star = UB$.
    \item While $\mathcal{Q}$ is not empty:
    \begin{enumerate}[label*=\arabic*., rightmargin=15pt, nolistsep]
        \item Select and remove problem $P$ from $\mathcal{Q}$.
        \item Compute a lower bound $LB$ for problem $P$ by solving its SDP relaxation.
        \item If $v^\star<\infty$ and  $(v^\star - LB)/v^\star \leq \varepsilon$, go to \textit{Step 5}.
        \item From the solution of the SDP relaxation, get a labeling $y^u$ and an upper bound $UB$ {by running the two-opt local search heuristic in Algorithm~\ref{alg:twoopt}}. If $UB < v^\star$, then set $v^\star \leftarrow UB$, $(y^u)^\star \leftarrow y^u$,  update $L_i$ and $U_i$ for $i=1,\ldots,n$, apply bound tightening, and go to \textit{Step 5.2}.
        \item Search for violated RLT inequalities. If any are found, add them to the current SDP relaxation, and go to \textit{Step 5.2}. If no violated RLT inequality is found, or the lower bound did not improve, go to \textit{Step 5.6}.
        \item Select an unlabeled data point with index $i$ and partition problem $P$ into two subproblems: one with $y_i=-1$ and the other with $y_i = 1$. Add each subproblem to $\mathcal{Q}$ and go to \textit{Step 5}.
    \end{enumerate}
\end{enumerate}
\textbf{Output}: Labeling vector $(y^u)^\star$ with objective value $v^\star$.

\end{algorithm}

{
\begin{remark}
Algorithm~\ref{alg:bbpseudocode} converges in a finite number of steps, even for $\varepsilon=0$, due to the finiteness of the branch-and-bound tree. Each branch-and-bound node corresponds to a partial labeling of the dataset, and its two children represent labelings of different unlabeled points. The root of the tree contains only the original set of labeled examples. Each leaf in the tree represents a complete labeling of the dataset. Consequently, the problem is convex at the leaf nodes. This implies that Algorithm~\ref{alg:bbpseudocode} terminates after exploring at most $2^{n-l}$ nodes, independently of the quality of lower and upper bounds. The effect of our machinery is to significantly reduce the number of nodes required to achieve nearly optimal solutions (where $\epsilon >0$), as demonstrated by the computational results in Section~\ref{sec:results}.
\end{remark}}


\section{Ideal Kernel}\label{sec:ideal}
When solving S3VM problems, the main goal is identifying the underlying ground-truth labels. The S3VM objective function may yield a solution far from the ground truth if the clustering assumption is not met or if hyperparameters are inadequately selected, as discussed by Chapelle et al.~\cite{chapelle2008optimization}. Therefore, an interesting question arises: is there a kernel function that ensures that the S3VM objective function aligns with the ground truth while simultaneously making the SDP relaxation tight?
The natural candidate for such a function is the ideal kernel, which establishes a class-based similarity score capable of perfectly discriminating between both classes. However, its analytical expression remains unknown and is dependent on the dataset. Assume that the classification outcome is available, i.e., we know 
$\gamma=[y_1,\ldots,y_l,y^{gt}_{l+1},\ldots,y_n^{gt}]^\top$, where $y_i^{gt}$ is the ground-truth label for the unlabeled data points. An example of an ideal kernel function is $\bar{K}=\gamma\gamma^\top$, where $\bar k(x_i,x_j)=-1 $ if $\gamma_i\not=\gamma_j$ and $\bar{k}(x_i,x_j)=1$ if $\gamma_i=\gamma_j$. This ensures that within-class distances in the feature space are zero and between-class distances are positive. Indeed, by applying the kernel trick, one can easily verify that
\begin{equation*}
\|\phi(x_i)-\phi(x_j)\|^2 = \bar k(x_i,x_i)-2\bar k(x_i,x_j)+\bar k(x_j,x_j)=\left\{\begin{array}{ll}
0     & \mbox{if } \gamma_i=\gamma_j, \\
 4    & \mbox{if } \gamma_i\not=\gamma_j.
\end{array}\right.
\end{equation*}
If we set $\bar K=\gamma\gamma^T$, we can prove that for any positive values of the penalty parameters $C_l$ and $C_u$, the global minimum of Problem~\eqref{eq:orig_prob} recovers the ground-truth classification. Furthermore, the SDP relaxation~\eqref{eq:SDP_basicrel} is tight. \\
\begin{proposition}\label{prop:ideal}
Given $C_l,C_u>0$, let $\bar{K}=\gamma\gamma^\top$, then the global minimum of Problem~\eqref{eq:orig_prob} is $v^\star=\gamma$, and $(\gamma,\gamma\gamma^\top)$ is the optimal solution of Problem~\eqref{eq:SDP_basicrel}.
\end{proposition}
\begin{proof}
Assume that $\bar{K}=\gamma\gamma^\top$. Then the cost matrix of Problem~\eqref{eq:orig_prob} can be explicitly computed by the Sherman-Morrison formula
\begin{equation*}
    C = \left(D+\gamma\gamma^\top \right)^{-1} = D^{-1}-\frac{D^{-1}\gamma\gamma^\top D^{-1}}{1+\gamma^\top D^{-1}\gamma},
\end{equation*}
where 
\begin{equation*}
    D=\begin{bmatrix}
    \frac{1}{2C_l}I_l & 0_{l \times (n-l)} \cr
       0_{(n-l) \times l}& \frac{1}{2C_u}  I_{n-l} 
\end{bmatrix}, \quad
D^{-1}=\begin{bmatrix}
    2C_lI_l & 0_{l \times (n-l)} \cr
       0_{(n-l) \times l}& 2C_u  I_{n-l}
\end{bmatrix}.
\end{equation*}
\noindent\newline
Indeed, the objective function of Problem~\eqref{eq:orig_prob} can be expressed as 
\begin{equation*}
    v^\top \left(D+\gamma\gamma^\top \right)^{-1}v = v^\top D^{-1}v^\top - \frac{v^\top D^{-1}\gamma\gamma^\top D^{-1}v}{1+\gamma^\top D^{-1}\gamma}.
\end{equation*}
By setting $z = \sqrt{D^{-1}}v$ and $\beta = \sqrt{D^{-1}}\gamma $, Problem~\eqref{eq:orig_prob} can be rewritten as:
\begin{equation}\label{eq:scaleprob}
\begin{aligned}
\min_{z} \quad & \|z\|^2-\frac{ (z^\top \beta)^2}{1+\beta^\top \beta} \\
\textrm{s.\,t.} \quad & y_i z_i \geq \sqrt{2C_l}, \qquad i = 1,\dots,l, \\
& z_i^2 \geq 2C_u, \qquad i = l + 1, \dots, n,\\
& z \in \mathbb{R}^n. \\
\end{aligned}
\end{equation}
Let $\lambda \in \mathbb{R}^n$ be the vector of multipliers associated with the inequality constraints. Then, we can write the (necessary) KKT conditions for Problem~\eqref{eq:scaleprob} as
\begin{align}\label{eq:KKTscaleprob1}
    2z_i-2\frac{z^\top \beta}{1+\|\beta\|^2}\beta -\lambda_iy_i &= 0, \quad i=1,\dots,l,\\ \label{eq:KKTscaleprob2}
    2z_i-2\frac{z^\top \beta}{1+\|\beta\|^2}\beta -2\lambda_iz_i &= 0, \quad i=l+1,\dots,n,\\\label{eq:KKTscaleprob3}
    \lambda_i\left(y_iz_i - \sqrt{2C_l}\right) &= 0, \quad i=1,\dots,l,\\  \label{eq:KKTscaleprob4}
    \lambda_i\left(z_i^2 - 2C_u,\right) &= 0, \quad i=l+1,\ldots,n,\\ \label{eq:KKTscaleprob5}
    \lambda_i &\ge 0, \quad i=1,\ldots,n. 
\end{align}
Now, consider the point $\hat z$ such that all the constraints in~\eqref{eq:scaleprob} are active, that is 
\begin{align*}
    \hat z_i & = y_i \sqrt{2C_l}, \quad i = 1,\dots,l \\
    \hat z_i^2 & = 2C_u, \quad i = l + 1,\dots, n.
\end{align*}
The second equation implies that either $\hat z_i=\sqrt{2C_u}$ or $\hat z_i=-\sqrt{2C_u}$ for $i=l+1,\ldots,n$, and we have that $\|\hat{z}\|^2=\|{\beta}\|^2$.
By~\eqref{eq:scaleprob}, the objective function value becomes
\begin{equation*}
\|\hat z\|^2-\frac{ (\beta^\top \hat z)^2}{1+\|\beta\|^2}  = \|\beta\|^2-\frac{(\sum_{i=1}^n\hat z_i\beta_i)^2}{1+\|\beta\|^2},
\end{equation*}
which is minimized when the term $\sum_{i=1}^n\hat z_i\beta_i$ is maximized. The term is maximal when $\hat z_i$ and $\beta_i$ have the same sign, meaning that $\hat z_i=y_i^{gt}\sqrt{2C_u}=\beta_i$. Then, the objective function value in $\hat z$ is 
\begin{equation*}
\|\hat z\|^2-\frac{ (\beta^\top \hat z)^2}{1+\|\beta\|^2} = \frac{ 1}{1+\beta^\top\beta}. 
\end{equation*}
By setting $\hat z = \beta$ in~\eqref{eq:KKTscaleprob1}--\eqref{eq:KKTscaleprob2}, we get 
\begin{align*}
    y_i\lambda_i = 2\beta_i-2\frac{\|\beta\|^2}{1+\|\beta\|^2}\beta_i, \quad  i=1,\ldots,l,\\
   2\beta_i^2-2\frac{\|\beta\|^2}{1+\beta^\top\beta}\beta_i^2 -2\lambda_i\beta_i^2=0, \quad i=l+1,\ldots,n,
\end{align*}
where the second equation is obtained by multiplying~\eqref{eq:KKTscaleprob2} by $\hat z_i$. Then, we have 
\begin{align*}
    \hat\lambda_i &= \frac{2\sqrt{2}C_l}{1+\|\beta\|^2}>0, \quad i=1,\ldots,l,\\
    \hat\lambda_i &= \frac{2\sqrt{2}C_u}{1+\|\beta\|^2}>0, \quad i=l+1,\ldots,n,
\end{align*}
which implies that the pair $(\hat z,\hat \lambda)$ is a KKT point. Looking at the objective function of Problem~\eqref{eq:scaleprob} and using the Cauchy-Schwartz inequality $(z^\top\beta)^2\le \|z\|^2\|\beta\|^2$, we get 
\begin{equation*}
\displaystyle f(z) = \|z\|^2-\frac{ (\beta^\top z)^2}{1+\beta^\top \beta}\ge  \frac{\|z\|^2}{1+\|\beta\|^2}.
\end{equation*}
At $\hat z$, we get $f(\hat z)=\frac{\|\hat z\|^2}{1+\|\beta\|^2}$, so the minimum is achieved, and the point is a global minimum of Problem~\eqref{eq:scaleprob}.

Finally, it remains to show that the SDP relaxation~\eqref{eq:our_basic} is tight. Plugging the ideal kernel into the objective function of~\eqref{eq:our_basic}, we get  
\begin{equation}\label{eq:SDP_basicrel_ideal}
\begin{aligned}
\min_{\bar{X}} \quad &  \left\langle  D^{-1}-\frac{D^{-1}\gamma\gamma^TD^{-1}}{1+\gamma^TD^{-1}\gamma}, X \right\rangle \\
\textrm{s.\,t.} \quad & L_i \leq x_i \leq U_i, \qquad i = 1,\dots,n, \\
& 1 \leq X_{ii} \leq \max\{ L_i^2, U_i^2 \}, \qquad i = 1,\dots,n, \\
& \bar{X} = \begin{pmatrix}
		X & x\\
		x^\top & 1\\
		\end{pmatrix} \succeq 0, \ x \in \mathbb{R}^n, \ X \in \mathbb{S}^n.
\end{aligned}
\end{equation}
We know that the point $(\gamma, \gamma\gamma^\top)$ is feasible for Problem~\eqref{eq:SDP_basicrel_ideal}, and that the objective function satisfies
\begin{equation*}
    \left \langle  D^{-1}-\frac{D^{-1}\gamma\gamma^\top D^{-1}}{1+\gamma^TD^{-1}\gamma}, \gamma\gamma^\top \right\rangle =\gamma D^{-1}\gamma -\frac{(\gamma^\top D^{-1}\gamma)^2 }{1+\gamma^\top D^{-1}\gamma} = \frac{\|\hat z\|^2}{1+\|\beta\|^2},
\end{equation*}
where the last equality derives from $\hat z = \beta = \sqrt{D^{-1}\gamma}$. Therefore, the relaxation is optimal.
\end{proof}

Although Proposition~\ref{prop:ideal} provides theoretical insights into the recovery properties of the SDP relaxation~\eqref{eq:our_basic}, the ideal kernel is not available in practice. However, the concept of an ideal kernel is useful because it defines how a kernel should be in order to achieve perfect classification. 
Indeed, we can expect a strong correlation between the objective function of the S3VM problem and the ground-truth labels, given that the kernel function is \enquote{close} to a kernel that effectively discriminates between the two classes.
In this setting, the SDP relaxation~\eqref{eq:our_basic} is expected to provide a strong bound that is further improved when RLT inequalities~\eqref{RLT} are added. 

The theoretical result linking bound quality to kernel effectiveness is supported by the numerical results in Section~\ref{sec:bbres}. Indeed, the only instances where we have huge gaps correspond to a wrong choice of the kernel function.

\section{Numerical Results}\label{sec:results}
This section provides the implementation details, outlines the details of the experimental setup, and thoroughly analyzes the computational results.

\subsection{Implementation details}
Our SDP-based B\&B algorithm is implemented in {\tt Julia} (version 1.9.4), and {\tt JuMP}~\cite{lubin2023jump} is used to formulate all occurring optimization problems. We use {\tt MOSEK} (version 10.1.8) to solve SDPs within our bounding routine and {\tt Gurobi} (version 10.0) to solve QCQPs for estimating box constraints. All experiments are performed on a machine with an Intel(R) Core(TM) i7-12700H processor with 14 cores clocked at 3.50 GHz, 16 GB of RAM, and Ubuntu 22.04 as the operating system. 
As for the cutting-plane approach, we add up to $5n$ of the most violated RLT inequalities at each iteration. In our numerical tests, we set the tolerance for checking the violation to $10^{-2}$. Furthermore, the tolerance for removing inactive cuts is set to $10^{-4}$. We stop the cutting-plane procedure when the relative difference between consecutive lower bounds is less than $10^{-3}$ and explore the B\&B tree using the best-first search strategy.

As an optimality measure, we use the percentage gap computed as $\varepsilon = ((UB - LB) / UB) \times 100$, where $UB$ is the best known upper bound and $LB$ is the lower bound. {Although our branch-and-bound algorithm is exact and capable of producing optimality certificates for any given $\varepsilon > 0$ in finite time, we stop our algorithm when $\varepsilon$ is less than $0.1\%$. This allows us to obtain nearly optimal solutions. Moreover, the remaining gap of $0.1\%$ is negligible in machine learning applications.} 

To enhance reproducibility, the source code has been made publicly available through the link \texttt{https://github.com/jschwiddessen/SDP-S3VM}. This repository provides all the tools to verify, build upon, and improve our global S3VM solver.

\subsection{Test instances}
To assess the effectiveness of our proposed branch-and-cut algorithm, we test it on various binary classification datasets. More precisely, we consider three artificial datasets, $\texttt{art\_100}$, $\texttt{art\_150}$, and $\texttt{2moons}$ and several real-world datasets taken from the literature\footnote{https://archive.ics.uci.edu/}. Table~\ref{tab:datasets} provides details on these datasets, including the number of data points ($n$), the number of features ($d$), and the total number of positive and negative examples, $(\# +1)$ and $(\# -1)$ respectively. We standardize the features such that each feature has zero mean and unit standard deviation. From each dataset, we derive different instances for our problem. Specifically, we randomly select a percentage $p \in \{10\%, 20\%, 30\%\}$ of data points to build the set of labeled examples {while preserving the percentage of samples for each class}. All the remaining points are treated as unlabeled. We consider three different random seeds for selecting the labeled data points to diversify the set of labeled examples. Indeed, the quality of the labeled data points plays a crucial role in guiding the optimization process and shaping the decision boundary between different classes. 


\begin{table}
\footnotesize
\centering
\caption{Characteristics of the tested datasets.}
\begin{tabular}{lcccc}
\toprule
Dataset   &  $n$   &  $d$ & $\# +1$ & $\# -1$ \\
\midrule
 $\texttt{art\_100}$ & 100 & 2&50 & 50\\ 
 $\texttt{art\_150}$ & 150 & 2& 77 & 73\\
 $\texttt{connectionist}$ & 208 & 60 & 111 & 97\\
 $\texttt{heart}$ &270 & 13 &120 &150\\
 $\texttt{2moons}$ & 300 &2 & 150& 150\\
 $\texttt{ionoshpere}$ & 351 & 33 &126 & 225\\
 $\texttt{PowerCons}$ & 360 & 144 & 180 & 180\\
 $\texttt{GunPoint}$ & 451 & 150 & 223 & 228\\
 $\texttt{arrhythmia}$ & 452 &191 &207 &245\\
 $\texttt{musk}$ & 476 & 166 & 207 & 269\\
 $\texttt{wdbc}$ & 569 & 30 & 357 & 212\\
\bottomrule
\end{tabular}
\label{tab:datasets}
\end{table}

\subsection{Lower bound comparison}\label{sec:bound_comp}
In this section, we conduct a comparative analysis between our lower bound presented in Section~\ref{sec:lb_bt} and the ones derived from the convex relaxations described in Section~\ref{sec:conv_rel}. Specifically, the following relaxations are considered:
\begin{enumerate}
    \item \textit{QP:} The standard QP relaxation obtained by solving Problem~\eqref{eq:orig_prob} without the non-convex constraints $v_i^2\ge 1$ for $i=1,\ldots,n$.
    \item \textit{QP-L:} The QP relaxation~\eqref{eq:qp_pen_bound} where violations of the non-convex constraints are penalized in the objective function by means of Lagrangian multipliers. The multipliers are estimated by solving the auxiliary SDP~\eqref{eq:qp_pen_bound_aux}.
    \item \textit{SDP:} The SDP relaxation~\eqref{eq:SDP_basicrel}.
    \item \textit{SDP-RLT:} The SDP relaxation~\eqref{eq:SDP_rel_RLT} without the balancing constraint~\eqref{eq:our_balancing}, where we apply bound tightening and add violated RLT cuts using our cutting-plane strategy. Box constraints $L_i \leq x_i \leq U_i$ are estimated by solving convex QCQPs of type~\eqref{eq:boundxi} for all $i=1, \dots, n$. The computation of a valid upper bound $UB$ required for Problem~\eqref{eq:boundxi} involves optimizing a supervised SVM only on the labeled data points and classifying the unlabeled data points. Thus, using this SVM labeling, $UB$ is determined as the solution of~\eqref{eq:convex_qp} without the balancing constraint~\eqref{eq:or_balancing}.
\end{enumerate}
Note that, for this set of experiments, we neglect the balancing constraint~\eqref{eq:our_balancing} to stick with the original convex relaxations described in the literature. Since incorporating the balancing constraint involves adding only a single linear constraint, we expect no differences in terms of bound comparison or computational time.
Computational results for the DNN presented in~\eqref{eq:DNN_large} are omitted due to long computing times and memory requirements, even for small-scale instances. Additionally, based on computational findings in~\cite{bai2016conic}, the DNN relaxation yields a marginally improved bound compared to the standard SDP but at significantly higher computational costs. Given that branch-and-bound algorithms depend on both strong and computationally efficient bounds, further computations using the DNN relaxation~\eqref{eq:DNN_large} are not pursued, as it is impractical to integrate it into any branch-and-bound algorithm. 

We select the radial basis function (RBF) kernel for all instances, defined as $K_{ij} = \textrm{exp}(-\gamma \|x_i - x_j\|^2)$, where $x_i, x_j \in \mathbb{R}^d$ represent feature vectors in the input space. Here, we set $\gamma=1/d$, $C_l = 1$, and $C_u=0.2 \cdot (l/(n-l)) \cdot C_l$. We average the results across the three seeds for each dataset and each percentage of labeled data points. Table~\ref{tab:bound} presents, for each relaxation, the number of labeled data points $l$, the number of unlabeled data points $n-l$, the percentage gap with respect to the upper bound $UB$, and the computational time in seconds. Additionally, for the SDP-RLT relaxation, Table~\ref{tab:bound} includes the \enquote{Time Box} column, indicating the average time spent in seconds for computing box constraints, and the \enquote{Iter} column, representing the average number of performed cutting-plane iterations.
Table~\ref{tab:bound} shows the superiority of the lower bound produced by SDP-RLT. While the bounds obtained from QP, QP-L, and SDP are faster to compute, they exhibit notably larger gaps. The substantial improvement in the lower bound achieved through RLT cuts justifies the time invested in computing box constraints. Across all instances, the average gap is less than 0.5\%, with only one instance exceeding 0.66\%. Notably, SDP-RLT attains a zero gap in 35 out of 81 instances and maintains gaps below 0.01\% in 46 instances (57\% of all instances). This indicates that we successfully solve more than half of the instances at the root node with this standard parameter configuration. 

\begin{sidewaystable}
\caption{For each relaxation, we present the average percentage gap (Gap [\%]) and the average computational time in seconds (Time [s]). This analysis is conducted for varying percentages of labeled data points, where $l$ is set to be 10\%, 20\%, and 30\% of the total number of data points, $n$. Additionally, for the SDP-RLT relaxation, we report the average time required for computing box constraints (Time Box [s]) and the average number of cutting-plane iterations (Iter).}
\label{tab:bound}
\begin{tabular}{lcc|cc|cc|cc|cccc}
\toprule
& & & \multicolumn{2}{c|}{QP} & \multicolumn{2}{c|}{QP-L} & \multicolumn{2}{c|}{SDP} & \multicolumn{3}{c}{SDP-RLT}\\
Instance	&	$l$	&	$n-l$	&	Gap [\%]	&	Time [s]	&	Gap [\%]	&	Time [s]	&	Gap [\%]	&	Time [s]	&	Time Box [s]	&	Gap [\%]	&	Time [s]	&	Iter\\
\midrule
2moons 	&	30	&	270	&	10.70	&	0.09	&	10.60	&	0.58	&	3.59	&	0.57	&	11.86	&	0.00	&	7.57	&	3.00	\\
2moons 	&	60	&	240	&	7.97	&	0.09	&	9.12	&	0.64	&	3.41	&	0.56	&	12.45	&	0.00	&	7.35	&	3.00	\\
2moons 	&	90	&	210	&	7.86	&	0.09	&	11.12	&	0.59	&	4.04	&	0.54	&	11.12	&	0.00	&	7.31	&	3.00	\\
arrhythmia 	&	44	&	408	&	19.83	&	0.13	&	10.27	&	1.95	&	9.81	&	1.74	&	45.98	&	0.42	&	115.82	&	6.00	\\
arrhythmia 	&	90	&	362	&	18.41	&	0.14	&	11.66	&	1.76	&	11.38	&	1.71	&	43.54	&	0.41	&	97.73	&	5.67	\\
arrhythmia 	&	135	&	317	&	16.57	&	0.14	&	11.08	&	1.79	&	10.69	&	1.77	&	41.20	&	0.22	&	61.59	&	5.00	\\
art100 	&	10	&	90	&	13.80	&	0.00	&	10.82	&	0.06	&	1.89	&	0.06	&	0.48	&	0.03	&	0.70	&	3.00	\\
art100 	&	20	&	80	&	14.17	&	0.00	&	16.31	&	0.06	&	1.12	&	0.06	&	0.51	&	0.00	&	0.58	&	3.00	\\
art100 	&	30	&	70	&	10.01	&	0.01	&	10.99	&	0.06	&	3.18	&	0.07	&	0.48	&	0.01	&	0.67	&	3.00	\\
art150 	&	14	&	136	&	14.01	&	0.05	&	11.80	&	0.18	&	3.15	&	0.13	&	1.30	&	0.04	&	1.44	&	3.00	\\
art150 	&	29	&	121	&	11.52	&	0.07	&	12.67	&	0.22	&	4.44	&	0.15	&	1.59	&	0.00	&	1.69	&	3.00	\\
art150 	&	44	&	106	&	10.19	&	0.05	&	16.36	&	0.18	&	5.31	&	0.15	&	1.32	&	0.01	&	1.38	&	3.00	\\
connectionist 	&	20	&	188	&	21.40	&	0.08	&	7.93	&	0.29	&	6.39	&	0.27	&	3.21	&	0.19	&	6.13	&	4.00	\\
connectionist 	&	41	&	167	&	19.27	&	0.08	&	10.17	&	0.33	&	8.40	&	0.27	&	3.09	&	0.16	&	9.84	&	4.67	\\
connection 	&	62	&	146	&	18.29	&	0.07	&	11.54	&	0.35	&	9.98	&	0.27	&	3.05	&	0.45	&	9.03	&	4.67	\\
GunPoint 	&	44	&	407	&	10.99	&	0.15	&	8.81	&	1.60	&	7.90	&	1.54	&	47.15	&	0.00	&	57.56	&	4.00	\\
GunPoint 	&	89	&	362	&	10.39	&	0.16	&	9.55	&	1.78	&	8.58	&	1.55	&	46.59	&	0.04	&	55.44	&	4.00	\\
GunPoint 	&	134	&	317	&	10.36	&	0.14	&	9.99	&	1.82	&	9.02	&	1.56	&	43.90	&	0.01	&	50.60	&	4.00	\\
heart 	&	27	&	243	&	19.68	&	0.07	&	8.06	&	0.45	&	5.84	&	0.39	&	6.92	&	0.22	&	10.36	&	4.00	\\
heart 	&	54	&	216	&	15.90	&	0.09	&	9.48	&	0.49	&	6.96	&	0.38	&	6.96	&	0.08	&	13.93	&	4.33	\\
heart 	&	81	&	189	&	14.57	&	0.08	&	9.79	&	0.50	&	7.65	&	0.38	&	6.37	&	0.15	&	12.05	&	4.33	\\
ionosphere 	&	34	&	317	&	19.90	&	0.12	&	7.31	&	1.06	&	3.85	&	0.91	&	19.84	&	0.66	&	19.53	&	3.67	\\
ionosphere 	&	70	&	281	&	16.67	&	0.17	&	6.22	&	0.98	&	3.87	&	0.93	&	19.67	&	0.01	&	20.73	&	3.33	\\
ionosphere 	&	104	&	247	&	16.30	&	0.13	&	7.27	&	0.98	&	4.97	&	0.90	&	17.98	&	0.00	&	27.77	&	4.00	\\
PowerCons 	&	36	&	324	&	18.83	&	0.13	&	13.31	&	1.10	&	5.52	&	0.95	&	21.80	&	0.04	&	22.79	&	3.67	\\
PowerCons 	&	72	&	288	&	17.04	&	0.11	&	14.66	&	1.00	&	6.50	&	0.88	&	19.12	&	0.01	&	26.26	&	4.00	\\
PowerCons	&	108	&	252	&	15.07	&	0.12	&	14.11	&	1.00	&	6.52	&	0.90	&	18.87	&	0.01	&	28.53	&	4.00	\\
\bottomrule
\end{tabular} 
\end{sidewaystable}

\subsection{Branch-and-cut results}\label{sec:bbres}
In this section, we report numerical experiments of the overall branch-and-cut approach. To the best of our knowledge, the only exact algorithms for L2-norm S3VMs are the branch-and-bound methods introduced in~\cite{chapelle2006branch} and~\cite{tian2017new}. Moreover, recent versions of Gurobi incorporate solvers for non-convex QCQPs and can solve problems of type~\eqref{eq:orig_prob_balancing}. We specifically focus on comparing against Gurobi, as the branch-and-bound method in~\cite{tian2017new} uses the relaxation in~\eqref{eq:qp_pen_bound} for the bounding routine, which has been demonstrated in Section~\ref{sec:bound_comp} to be considerably weaker than our SDP-RLT relaxation. Moreover, the branch-and-bound approach proposed in~\cite{chapelle2006branch} is only suitable for small instances ($n \leq 50$) since it relies on analytic bounds that are notably weak and require the exploration of a large portion of the branch-and-bound tree.

Following the related literature, we perform a $k$-fold cross-validation to determine the hyperparameters on only the set of labeled data points. By cross-validation, we choose the kernel function and the penalty parameter $C_l$ for labeled data points. We set $k=10$ for all datasets, except for $\texttt{art100}$ and $\texttt{art150}$, where we set $k=2$ due to the limited number of labeled data points. The penalty parameter for unlabeled data points $C_u$ is set to $0.2 \cdot (l/(n-l)) \cdot C_l$. The factor $l/(n-l)$ balances the contribution of labeled and unlabeled data points, and we further diminish the influence of unlabeled data points by the factor $0.2$ to give more weight to labeled data points.

The first set of experiments aims to understand how Gurobi behaves on different instances. We restrict to the smaller ones, that is, $n\leq 300$, and we set a time limit of 3600 seconds. In Table~\ref{tab:gurobi_avg}, we report the instance name, the number of labeled ($l$), and unlabeled data points ($n-l$). Then, for both Gurobi and our solver (SDP-S3VM), we report the average percentage gap (Gap [\%]), the average computational time in seconds (Time [s]), and the number of instances solved to global optimality (Solved), ranging between 0 and 3. The computational results show that Gurobi encounters difficulties in closing the gap on most instances. Specifically, only 11 out of 45 instances are solved to optimality. As the number of unlabeled data points increases, the instances become more challenging, and the optimality gap becomes significantly larger after one hour of computation. On the other hand, our solver consistently outperforms Gurobi while solving 43 out of 45 instances and producing an average gap smaller than $1\%$ on the two remaining instances of $\texttt{connectionist}$.

\begin{table}[!ht]
\footnotesize
\caption{Average gap and computational time for instances with $n\le 300$ solved by Gurobi and our solver SDP-S3VM. For each instance, the gap and time are averaged over three different seeds with an increasing number of labeled data points ($p=10\%, 20\%, 30\%$). We also report in column \enquote{Solved} the number of instances solved for a given dataset and the value of $p$.}
\label{tab:gurobi_avg}
\begin{tabular}{lcc|ccc|ccc}
\toprule
& & & \multicolumn{3}{c|}{Gurobi} & \multicolumn{3}{c}{SDP-S3VM}\\
Instance	&	$l$	&	$n-l$	& Gap [\%]	& Time [s] & Solved & Gap [\%]	& Time [s] & Solved\\
\midrule
art100 	&	10	&	90	&	7.37	&	3600	&	0	&	0.10	&	26.11	&	3	\\
art100	&	20	&	80	&	3.09	&	2467.43	&	1	&	0.10	&	13.28	&	3	\\
art100 	&	30	&	70	&	3.27	&	2401.26	&	1	&	0.10	&	37.48	&	3	\\
art150 	&	14	&	136	&	8.44	&	3600	&	0	&	0.10	&	61.05	&	3	\\
art150	&	29	&	121	&	2.72	&	1450.20	&	2	&	0.10	&	1.89	&	3	\\
art150 	&	44	&	106	&	2.52	&	2629.13	&	1	&	0.10	&	2.44	&	3	\\
connectionist 	&	20	&	188	&	16.83	&	3600	&	0	&	0.88	&	2587.20	&	1	\\
connectionist 	&	62	&	146	&	12.87	&	3600	&	0	&	0.10	&	248.07	&	3	\\
connectionist	&	41	&	167	&	10.71	&	3600	&	0	&	0.10	&	104.95	&	3	\\
heart 	&	27	&	243	&	14.00	&	3600	&	0	&	0.10	&	38.89	&	3	\\
heart 	&	54	&	216	&	10.21	&	3600	&	0	&	0.10	&	64.45	&	3	\\
heart 	&	81	&	189	&	10.58	&	3600	&	0	&	0.10	&	16.22	&	3	\\
2moons 	&	30	&	270	&	6.52	&	3600	&	0	&	0.10	&	16.22	&	3	\\
2moons 	&	60	&	140	&	0.03	&	1023.52	&	3	&	0.10	&	22.07	&	3	\\
2moons 	&	90	&	210	&	0.05	&	1.95	&	3	&	0.10	&	21.50	&	3	\\
\bottomrule
\end{tabular} 
\end{table}

In Tables~\ref{tab:small_full} and~\ref{tab:large_full}, we report the detailed performance of our solver on the small and large instances, respectively. Specifically, we show the performance of our B\&B approach, including the number of nodes (Nodes), the final gap (Gap [\%]), and computational time (Time [s]). We also evaluate the quality of the produced solution in terms of accuracy. By accuracy, we mean the percentage of correctly labeled data points among the unlabeled ones. Indeed, we report the number of labeled ($l$) and unlabeled data points ($n-l$), the kernel, the penalty parameter for the labeled data points $C_l$ ($C_u = 0.2 \cdot (l/(n-l)) \cdot C_l$), the percentage gap (Gap [\%], the number of nodes (Nodes), and the computational time in seconds. Furthermore, we report the accuracy on all unlabeled data points produced by our solution and the accuracy of the baseline, which is the accuracy on unlabeled data points obtained by a supervised SVM built using only labeled data points as a training set. {We do not report the time required by the supervised SVM, as only one convex QP needs to be solved in this case, and the solution time is negligible even compared to the execution of an S3VM lower bound procedure.}
We set a time limit of 3600 seconds (one hour) on the small instances ($n\le 300$) and a time limit of 10800 seconds (three hours) on the larger ones.


The computational results show that we solve 88 instances out of 99 to the required precision (gap less than 0.1\%). On the remaining 11 instances, we hit the time limit. Overall, our method performs very well, with a low number of average nodes and closing many instances at the root node (39 out of 99 instances), confirming the high quality of our SDP bound. For the two instances of \texttt{connectionist} that we do not solve, we achieve a gap of 1.25\% and 1.29\% after one hour of computation. Among the large-scale instances, we do not solve nine instances, but we consistently achieve a gap of less than 2\% except in four cases. These four instances include two from the \texttt{arrhythmia} dataset and two from the \texttt{musk} dataset. However, we can see that in all four instances, the cross-validation leads to a wrong choice of the kernel. Indeed, the best choice of the kernel was the RBF (chosen on all the other instances derived by the two datasets \texttt{arrhythmia} and \texttt{musk}), whereas on all four instances, the linear kernel was chosen. This can be explained by the small number of labeled data points used for the cross-validation, which may lead to wrong choices. We will see in the following subsection that when choosing the RBF kernel, the gap decreases to values less than 1\%, and the accuracy increases.

These examples show that solving an S3VM instance to optimality is influenced by various factors, such as the number of unlabeled data points, the chosen hyperparameters, and the kernel selection due to the potential violation of the clustering assumption for certain kernel choices~\cite{chapelle2005semi}. Regardless of the challenges posed by solving a particular instance from a mathematical optimization perspective, the accuracy of the final labeling produced by the global minimum is influenced by the aforementioned choices. 
Further discussion on this aspect will be provided in the next section.

{\footnotesize
\begin{longtable}{lcc|cc|cccc|c}
\caption{In this table, we show the performance of our solver on small instances. We report the instance name, the number of labeled data points $l$, the number of unlabeled data points $n-l$, the kernel, the penalty parameter $C_l$, the percentage gap (Gap [\%]), the number of nodes (Nodes), the computational time in seconds (Time [s]), the accuracy of the produced solution on the unlabeled data points (Acc. [\%]), the accuracy of a supervised SVM built on the labeled data points and tested on the unlabeled ones (Base [\%]). We highlight in boldface the labeling that achieves the highest accuracy on the unlabeled data points.}
\label{tab:small_full}\\
\toprule
Instance	&	$l$	&	$n-l$	& Kernel & $C_l$ & Gap [\%] & Nodes & Time [s] & Acc. [\%] & Base [\%]\\  
\midrule
art100 	&	10	&	90	&	 linear 	&	0.1	&	0.1	&	3	&	11.39	&	\textbf{94.44}	&	\textbf{94.44}	\\
art100 	&	10	&	90	&	 linear 	&	0.1	&	0.1	&	11	&	13.47	&	95.56	&	\textbf{96.67}	\\
art100 	&	10	&	90	&	 RBF 	&	0.1	&	0.1	&	185	&	53.48	&	\textbf{91.11}	&	87.78	\\
art100 	&	20	&	80	&	 linear 	&	0.1	&	0.1	&	1	&	10.65	&	\textbf{96.25}	&	93.75	\\
art100 	&	20	&	80	&	 linear 	&	0.1	&	0.1	&	1	&	10.68	&	\textbf{97.50}	&	93.75	\\
art100 	&	20	&	80	&	 RBF 	&	0.1	&	0.1	&	27	&	18.52	&	\textbf{90.00}	&	87.50	\\
art100 	&	30	&	70	&	 linear 	&	0.1	&	0.1	&	1	&	10.77	&	\textbf{95.71}	&	92.86	\\
art100 	&	30	&	70	&	 RBF 	&	0.1	&	0.1	&	3	&	11.00	&	\textbf{94.29}	&	90.00	\\
art100 	&	30	&	70	&	 RBF 	&	0.1	&	0.1	&	53	&	26.54	&	\textbf{94.29}	&	87.14	\\\midrule
art150 	&	14	&	136	&	 RBF 	&	1.58	&	0.1	&	115	&	101.13	&	84.56	&	\textbf{87.50}	\\
art150 	&	14	&	136	&	 linear 	&	0.1	&	0.1	&	1	&	2.40	&	\textbf{91.91}	&	90.44	\\
art150 	&	14	&	136	&	 RBF 	&	0.1	&	0.1	&	11	&	8.92	&	\textbf{86.76}	&	83.82	\\
art150 	&	29	&	121	&	 RBF 	&	3.16	&	0.1	&	1	&	2.88	&	\textbf{89.26}	&	88.43	\\
art150 	&	29	&	121	&	 RBF 	&	0.1	&	0.1	&	219	&	168.61	&	\textbf{90.91}	&	87.60	\\
art150 	&	29	&	121	&	 RBF 	&	3.98	&	0.1	&	9	&	11.67	&	85.12	&	\textbf{86.78}	\\
art150 	&	44	&	106	&	 linear 	&	0.1	&	0.1	&	1	&	1.72	&	\textbf{88.68}	&	\textbf{88.68}	\\
art150 	&	44	&	106	&	 linear 	&	0.1	&	0.1	&	1	&	1.56	&	\textbf{89.62}	&	\textbf{89.62}	\\
art150 	&	44	&	106	&	 linear 	&	0.2	&	0.1	&	1	&	2.40	&	\textbf{85.85}	&	83.02	\\\midrule
connectionist 	&	20	&	188	&	 linear 	&	0.10	&	1.25	&	1448	&	3600	&	59.04	&	\textbf{60.11}	\\
connectionist 	&	20	&	188	&	 linear 	&	0.10	&	0.1	&	319	&	559.59	&	65.43	&	\textbf{66.49}	\\
connectionist 	&	20	&	188	&	 RBF 	&	1.26	&	1.29	&	1140	&	3600	&	69.15	&	\textbf{69.68}	\\
connectionist 	&	41	&	167	&	 RBF 	&	0.50	&	0.1	&	3	&	16.13	&	71.86	&	\textbf{75.45}	\\
connectionist 	&	41	&	167	&	 RBF 	&	0.10	&	0.1	&	263	&	518.69	&	59.88	&	\textbf{62.87}	\\
connectionist 	&	41	&	167	&	 linear 	&	0.10	&	0.1	&	129	&	209.38	&	\textbf{74.25}	&	73.65	\\
connectionist 	&	62	&	146	&	 RBF 	&	1.26	&	0.1	&	41	&	108.72	&	86.3	&	\textbf{86.99}	\\
connectionist 	&	62	&	146	&	 RBF 	&	0.63	&	0.1	&	77	&	188.95	&	71.23	&	\textbf{71.92}	\\
connectionist 	&	62	&	146	&	 RBF 	&	0.50	&	0.1	&	3	&	17.19	&	\textbf{86.30}	&	80.82	\\\midrule
heart 	&	27	&	243	&	 RBF 	&	0.10	&	0.1	&	3	&	23.12	&	\textbf{81.89}	&	80.25	\\
heart 	&	27	&	243	&	 RBF 	&	0.79	&	0.1	&	1	&	20.08	&	\textbf{80.25}	&	79.42	\\
heart 	&	27	&	243	&	 linear 	&	0.16	&	0.1	&	21	&	73.47	&	69.96	&	\textbf{70.37}	\\
heart 	&	54	&	216	&	 RBF 	&	0.10	&	0.1	&	1	&	11.49	&	\textbf{83.33}	&	82.41	\\
heart 	&	54	&	216	&	 RBF 	&	0.40	&	0.1	&	1	&	20.79	&	\textbf{84.72}	&	82.87	\\
heart 	&	54	&	216	&	 linear 	&	0.10	&	0.1	&	47	&	161.07	&	\textbf{76.85}	&	75.93	\\
heart 	&	81	&	189	&	 RBF 	&	0.10	&	0.1	&	1	&	11.75	&	\textbf{85.19}	&	\textbf{85.19}	\\
heart 	&	81	&	189	&	 RBF 	&	0.10	&	0.1	&	1	&	18.80	&	\textbf{85.19}	&	\textbf{85.19}	\\
heart 	&	81	&	189	&	 RBF 	&	0.10	&	0.1	&	1	&	18.12	&	\textbf{84.13}	&	83.07	\\\midrule
2moons 	&	30	&	270	&	 RBF 	&	2.51	&	0.1	&	1	&	33.71	&	\textbf{97.78}	&	97.04	\\
2moons 	&	30	&	270	&	 RBF 	&	0.10	&	0.1	&	1	&	29.25	&	\textbf{87.04}	&	\textbf{87.04}	\\
2moons 	&	30	&	270	&	 RBF 	&	0.50	&	0.1	&	1	&	30.10	&	91.85	&	\textbf{92.22}	\\
2moons 	&	60	&	240	&	 RBF 	&	3.98	&	0.1	&	1	&	21.77	&	\textbf{100.00}	&	\textbf{100.00}	\\
2moons 	&	60	&	240	&	 RBF 	&	3.98	&	0.1	&	1	&	21.62	&	\textbf{100.00}	&	\textbf{100.00}	\\
2moons 	&	60	&	240	&	 RBF 	&	1.26	&	0.1	&	1	&	22.83	&	\textbf{97.92}	&	\textbf{97.92}	\\
2moons 	&	90	&	210	&	 RBF 	&	3.16	&	0.1	&	1	&	21.38	&	\textbf{100.00}	&	\textbf{100.00}	\\
2moons 	&	90	&	210	&	 RBF 	&	2.51	&	0.1	&	1	&	21.15	&	\textbf{100.00}	&	\textbf{100.00}	\\
2moons 	&	90	&	210	&	 RBF 	&	1.26	&	0.1	&	1	&	21.97	&	\textbf{97.62}	&	\textbf{97.62}	\\
\bottomrule
\end{longtable}}


{\footnotesize
\begin{longtable}{lcc|cc|cccc|c}
\caption{In this table, we show the performance of our solver on large instances. We report the instance name, the number of labeled data points $l$, the number of unlabeled data points $n-l$, the kernel, the penalty parameter $C_l$, the percentage gap (Gap [\%]), the number of nodes (Nodes), the computational time in seconds (Time [s]), the accuracy of the produced solution on the unlabeled data points (Acc. [\%]), the accuracy of a supervised SVM built on the labeled data points and tested on the unlabeled ones (Base [\%]). We highlight in boldface the labeling that achieves the highest accuracy on the unlabeled data points.}
\label{tab:large_full}\\
\toprule
Instance	&	$l$	&	$n-l$	& Kernel & $C_l$ & Gap [\%] & Nodes & Time [s] & Acc. [\%] & Base [\%]\\  
\midrule
ionosphere 	&	34	&	317	&	 RBF 	&	0.79	&	0.1	&	59	&	529.48	&	\textbf{91.80}	&	81.70	\\
ionosphere 	&	34	&	317	&	 linear 	&	0.40	&	0.1	&	73	&	492.74	&	88.33	&	\textbf{88.96}	\\
ionosphere 	&	34	&	317	&	 linear 	&	0.13	&	0.1	&	3	&	50.05	&	\textbf{87.38}	&	84.23	\\
ionosphere 	&	70	&	281	&	 RBF 	&	2.51	&	0.1	&	3	&	107.89	&	\textbf{90.75}	&	90.04	\\
ionosphere 	&	70	&	281	&	 RBF 	&	10.00	&	0.1	&	7	&	181.61	&	\textbf{91.46}	&	85.05	\\
ionosphere 	&	70	&	281	&	 linear 	&	0.16	&	0.1	&	1	&	43.55	&	\textbf{88.61}	&	87.54	\\
ionosphere 	&	104	&	247	&	 RBF 	&	2.51	&	0.1	&	5	&	128.45	&	90.28	&	\textbf{90.69}	\\
ionosphere 	&	104	&	247	&	 linear 	&	0.32	&	0.1	&	37	&	221.45	&	\textbf{88.26}	&	86.64	\\
ionosphere 	&	104	&	247	&	 linear 	&	0.16	&	0.1	&	1	&	56.87	&	89.47	&	\textbf{90.69}	\\
\midrule
PowerCons 	&	36	&	324	&	 RBF 	&	1.26	&	0.1	&	11	&	139.97	&	\textbf{95.06}	&	93.83	\\
PowerCons 	&	36	&	324	&	 RBF 	&	1.58	&	0.1	&	1	&	45.20	&	95.37	&	\textbf{96.30}	\\
PowerCons 	&	36	&	324	&	 linear 	&	0.10	&	0.1	&	53	&	534.19	&	\textbf{97.84}	&	94.44	\\
PowerCons	&	72	&	288	&	 RBF 	&	0.16	&	0.1	&	11	&	101.41	&	\textbf{95.83}	&	94.79	\\
PowerCons 	&	72	&	288	&	 RBF 	&	1.58	&	0.1	&	1	&	30.79	&	96.53	&	\textbf{97.57}	\\
PowerCons 	&	72	&	288	&	 linear 	&	0.10	&	0.1	&	55	&	375.76	&	\textbf{98.61}	&	97.57	\\
PowerCons 	&	108	&	252	&	 linear 	&	0.10	&	0.1	&	11	&	129.53	&	\textbf{98.81}	&	\textbf{98.81}	\\
PowerCons 	&	108	&	252	&	 linear 	&	0.10	&	0.1	&	15	&	109.83	&	98.81	&	\textbf{99.21}	\\
PowerCons 	&	108	&	252	&	 linear 	&	0.10	&	0.1	&	17	&	169.85	&	98.41	&	\textbf{99.21}	\\
\midrule
GunPoint 	&	44	&	407	&	 RBF 	&	10.00	&	0.1	&	1	&	94.24	&	87.71	&	\textbf{88.21}	\\
GunPoint 	&	44	&	407	&	 RBF 	&	0.79	&	0.1	&	207	&	2925.41	&	84.03	&	\textbf{87.96}	\\
GunPoint 	&	44	&	407	&	 RBF 	&	0.79	&	0.1	&	7	&	318.37	&	\textbf{88.94}	&	86.98	\\
GunPoint 	&	89	&	362	&	 RBF 	&	3.98	&	0.1	&	1	&	130.63	&	95.58	&	\textbf{96.13}	\\
GunPoint 	&	89	&	362	&	 RBF 	&	7.94	&	0.1	&	1	&	138.84	&	\textbf{95.03}	&	94.75	\\
GunPoint 	&	89	&	362	&	 RBF 	&	3.16	&	0.1	&	1	&	98.54	&\textbf{93.09}	&	92.54	\\
GunPoint 	&	134	&	317	&	 RBF 	&	3.16	&	0.1	&	1	&	88.93	&	95.58	&	\textbf{96.53}	\\
GunPoint 	&	134	&	317	&	 RBF 	&	5.01	&	0.1	&	1	&	60.82	&	94.01	&	\textbf{94.32}	\\
GunPoint 	&	134	&	317	&	 RBF 	&	10.00	&	0.1	&	23	&	484.11	&	91.80	&	\textbf{92.43}	\\
\midrule
arrhythmia 	&	44	&	408	&	 RBF 	&	0.63	&	0.1	&	67	&	2463.22	&	\textbf{72.30}	&	65.93	\\
arrhythmia 	&	44	&	408	&	 RBF 	&	0.20	&	0.1	&	1	&	88.19	&	\textbf{72.06}	&	63.73	\\
arrhythmia 	&	44	&	408	&	 linear 	&	0.10	&	9.5	&	563	&	10800	&	69.12	&	\textbf{69.61}	\\
arrhythmia 	&	90	&	362	&	 RBF 	&	1.58	&	0.37	&	555	&	10800	&	\textbf{74.31}	&	70.72	\\
arrhythmia 	&	90	&	362	&	 linear 	&	0.32	&	10.91	&	532	&	10800	&	\textbf{67.40}	&	67.13	\\
arrhythmia 	&	90	&	362	&	 RBF 	&	1.26	&	0.18	&	371	&	10800	&	\textbf{70.99}	&	\textbf{70.99}	\\
arrhythmia 	&	135	&	317	&	 RBF 	&	2.00	&	0.56	&	632	&	10800	&	\textbf{75.08}	&	74.45	\\
arrhythmia 	&	135	&	317	&	 RBF 	&	0.25	&	0.1	&	79	&	2002.53	&	\textbf{74.45}	&	72.87	\\
arrhythmia 	&	135	&	317	&	 RBF 	&	3.98	&	1.88	&	415	&	10800	&	72.24	&	\textbf{74.13}	\\
\midrule
musk 	&	46	&	430	&	 linear 	&	0.79	&	30.75	&	323	&	10800	&	\textbf{70.23}	&	69.30	\\
musk 	&	46	&	430	&	 RBF 	&	0.25	&	0.1	&	1	&	235.84	&	68.37	&	\textbf{72.33}	\\
musk 	&	46	&	430	&	 linear 	&	0.10	&	9.11	&	503	&	10800	&	\textbf{67.44}	&	66.51	\\
musk 	&	94	&	382	&	 RBF 	&	1.00	&	0.16	&	455	&	10800	&	80.63	&	\textbf{81.68}	\\
musk 	&	94	&	382	&	 RBF 	&	0.50	&	0.1	&	3	&	161.91	&	\textbf{81.94}	&	81.68	\\
musk 	&	94	&	382	&	 RBF 	&	0.25	&	0.14	&	421	&	10800	&	\textbf{78.53}	&	77.75	\\
musk 	&	142	&	334	&	 RBF 	&	2.00	&	0.1	&	265	&	4173.31	&	\textbf{88.02}	&	85.03	\\
musk 	&	142	&	334	&	 RBF 	&	6.31	&	1.16	&	583	&	10800	&	\textbf{88.62}	&	87.13	\\
musk 	&	142	&	334	&	 RBF 	&	0.79	&	0.1	&	645	&	7914.07	&	\textbf{84.43}	&	83.83	\\
\midrule
wdbc 	&	56	&	513	&	 linear 	&	0.25	&	0.1	&	69	&	1398.81	&	\textbf{96.30}	&	95.32	\\
wdbc 	&	56	&	513	&	 linear 	&	0.10	&	0.1	&	3	&	210.55	&	94.74	&	\textbf{94.93}	\\
wdbc 	&	56	&	513	&	 linear 	&	0.10	&	0.1	&	1	&	193.61	&	\textbf{96.88}	&	\textbf{96.88}	\\
wdbc 	&	113	&	456	&	 RBF 	&	1.26	&	0.1	&	1	&	281.77	&	\textbf{96.49}	&	96.05	\\
wdbc 	&	113	&	456	&	 linear 	&	0.10	&	0.1	&	15	&	438.33	&	94.96	&	\textbf{95.83}	\\
wdbc 	&	113	&	456	&	 RBF 	&	5.01	&	0.1	&	11	&	381.48	&	\textbf{96.05}	&	\textbf{96.05}	\\
wdbc 	&	170	&	399	&	 RBF 	&	3.98	&	0.1	&	11	&	543.02	&	96.74	&	\textbf{97.49}	\\
wdbc 	&	170	&	399	&	 RBF 	&	0.32	&	0.1	&	1	&	279.65	&	94.49	&	\textbf{94.99}	\\
wdbc 	&	170	&	399	&	 RBF 	&	3.16	&	0.1	&	3	&	163.32	&	\textbf{96.99}	&	96.49	\\
\bottomrule
\end{longtable}}

\subsection{Machine learning perspective}\label{sec:ml}
In this section, we discuss the obtained results from the machine learning point of view. Our method solves S3VM classification problems to global optimality, thus avoiding the issues related to the presence of local minima. However, this may not be enough to obtain good classification performance. Semi-supervised learning primarily aims to exploit unlabeled data to construct better classification algorithms. As it turns out, this is not always easy or even possible. As mentioned earlier, unlabeled data is only useful if it carries information for label prediction that is not contained in the labeled data alone or cannot be easily extracted from it. Moreover, we also remark that the clustering assumption is a crucial and necessary condition for semi-supervised learning to succeed. In other words, if the data points (both unlabeled and labeled) cannot be meaningfully clustered, it is impossible for a semi-supervised learning method to improve on a supervised learning one.

However, assuming that the problem is solved to global optimality, some issues can be fixed by either considering the balancing constraint or changing the hyperparameters. To begin, we illustrate the advantage of including the balancing constraint with an example, especially when the labeled data points lack sufficient information. Indeed, depending on the quality of labeled data points, fluctuation in accuracy can be experienced when minimizing the objective function of Problem~\eqref{eq:orig_prob}, and the global minimum may perform poorly in terms of accuracy {with respect to} local minima with higher objective value. This can result from the inclination to cluster as many closely located points together as possible. Therefore, if the labeled data points are not informative enough, this may lead to collocating many unlabeled points in the same class, leading to poor accuracy.

\begin{figure}[!ht]
\begin{subfigure}{.5\linewidth}
\begin{tikzpicture}
\begin{axis}[
    title={},
      width=150pt,
      height=140pt,
      xtick=\empty,
      ytick=\empty,
      legend pos=south east,
      legend style={nodes={scale=0.5}}
]
\addplot[scatter/classes={1={blue}, -1={red}}, scatter, mark=*, only marks, scatter src=explicit symbolic, mark size=1.5pt]
table[x=x,y=y,meta=label]{art_gt.txt};
\legend{$+1$, $-1$}
\end{axis}
\end{tikzpicture}
\centering
\caption{Ground-truth classification} \label{fig:1a}
\end{subfigure}%
\begin{subfigure}{.5\linewidth}
\begin{tikzpicture}
\begin{axis}[
    title={},
      width=150pt,
      height=140pt,
      xtick=\empty,
      ytick=\empty,
      legend pos=south east,
      legend style={nodes={scale=0.5}}
]
\addplot[scatter/classes={1={mark=*, blue}, -1={mark=*, red}, 0={mark=o, black!40}}, scatter, only marks, scatter src=explicit symbolic, mark size=1.5pt]
table[x=x,y=y,meta=used]{art_gt.txt};
\legend{Label $+1$, Label $-1$, Unlabeled}
\end{axis}
\end{tikzpicture}
\centering
\caption{Labeled and unlabeled data points} \label{fig:1b}
\end{subfigure}%

\begin{subfigure}{.5\linewidth}
\begin{tikzpicture}
\begin{axis}[
    title={},
      width=150pt,
      height=140pt,
      xtick=\empty,
      ytick=\empty,
      legend pos=south east,
      legend style={nodes={scale=0.5}}
]
\addplot[scatter/classes={1={mark=+, blue}, -1={mark=+, red}}, scatter, only marks, scatter src=explicit symbolic, mark size=1.5pt]
table[x=x,y=y,meta=global]{art_gt.txt};
\legend{Label $+1$, Label $-1$}
\end{axis}
\end{tikzpicture}
\centering
\caption{Optimal S3VM solution} \label{fig:1c}
\end{subfigure}%
\begin{subfigure}{.5\linewidth}
\begin{tikzpicture}
\begin{axis}[
    title={},
      width=150pt,
      height=140pt,
      xtick=\empty,
      ytick=\empty,
      legend pos=south east,
      legend style={nodes={scale=0.5}}
]
\addplot[scatter/classes={1={mark=+, blue}, -1={mark=+, red}}, scatter, only marks, scatter src=explicit symbolic, mark size=1.5pt]
table[x=x,y=y,meta=balance]{art_gt.txt};
\legend{Label $+1$, Label $-1$}
\end{axis}
\end{tikzpicture}
\centering
\caption{Optimal S3VM solution with balancing} \label{fig:1d}
\end{subfigure}%
\end{figure}

As an example, consider Figures~\ref{fig:1a}--\ref{fig:1d}. We have a small dataset of $50$ data points in the plane, and the ground truth is reported in Figure~\ref{fig:1a}. Assume we have 10\% of labeled data points that are the ones shown in Figure~\ref{fig:1b}. Then, if we solve Problem~\eqref{eq:orig_prob} to global optimality, we get the labeling of Figure~\ref{fig:1c}, which achieves an accuracy of 52\% on the unlabeled data points. However, if we introduce the balancing constraint, the global minimum leads to the labeling in Figure~\ref{fig:1d} that perfectly recovers the ground truth (100\% accuracy).
{The example in Figure~\ref{fig:1a} might be seen as an extreme case, but our experience shows that omitting the balancing constraint~\eqref{eq:our_balancing} deteriorates the performance, even on real-world data. To demonstrate this behavior, we select two instances with less balanced class distribution, \texttt{ionosphere} and \texttt{heart}, and compare the results with and without the balancing constraint. Table~\ref{tab:bal} shows the accuracy on the unlabeled data points without the balancing constraint (column No Balancing [\%]) and with it (column Balancing [\%]). The hyperparameters used are consistent with those in Section~\ref{sec:bbres}. We also include the Baseline column from Tables~\ref{tab:small_full} and~\ref{tab:large_full} for comparison. The results indicate that the balancing constraint significantly improves the accuracy of the globally optimal solution, outperforming the S3VM without this constraint in most cases. Specifically, on the \texttt{ionosphere} dataset, we observe an average improvement of 4.16\%, with individual instances showing impressive accuracy gains (up to 17.98\%). Additionally, for the \texttt{heart} dataset, the average accuracy without the balancing constraint falls below the baseline, whereas including the constraint raises it above the baseline.}

\begin{table}[!ht]
\footnotesize
\centering
\caption{{We report the instance name, the number of labeled data points $l$, the number of unlabeled data points $n-l$, the accuracy of the S3VM solution without the balancing constraint (No balancing [\%]), the accuracy of the S3VM solution with the balancing constraint (Balancing [\%]), the accuracy of a supervised SVM built on the labeled data points and tested on the unlabeled ones (Baseline [\%]). We highlight in boldface the labeling that achieves the highest accuracy on the unlabeled data points. Additionally, we report the average accuracy for each method.}}
\label{tab:bal}
\begin{tabular}{lcc|ccc}
\toprule
Instance	&	$l$	&	$n-l$	& No balancing [\%]	& Balancing [\%] & Baseline [\%]\\
\midrule
ionosphere	&	34	&	317	&	73.82	&	\textbf{91.80}	&	81.70	\\
ionosphere	&	34	&	317	&	\textbf{89.91}	&	88.33	&	88.96	\\
ionosphere	&	34	&	317	&	85.17	&	\textbf{87.38}	&	84.23	\\
ionosphere	&	70	&	281	&	90.04	&	\textbf{90.75}	&	90.04	\\
ionosphere	&	70	&	281	&	85.05	&	\textbf{91.46}	&	85.05	\\
ionosphere	&	70	&	281	&	87.19	&	\textbf{88.61}	&	87.54	\\
ionosphere	&	104	&	247	&	\textbf{90.69}	&	90.28	&	\textbf{90.69}	\\
ionosphere	&	104	&	247	&	85.43	&	\textbf{88.26}	&	86.64	\\
ionosphere	&	104	&	247	&	89.88	&	89.47	&	\textbf{90.69}	\\
\midrule
Average	&		&		&	85.23	&	\textbf{89.36}	&	86.31	\\
\midrule
heart	&	27	&	243	&	77.37	&	\textbf{81.89}	&	80.25	\\
heart	&	27	&	243	&	\textbf{80.66}	&	80.25	&	79.42	\\
heart	&	27	&	243	&	\textbf{71.19}	&	69.96	&	70.37	\\
heart	&	54	&	216	&	81.94	&	\textbf{83.33}	&	82.41	\\
heart	&	54	&	216	&	81.94	&	\textbf{84.72}	&	82.87	\\
heart	&	54	&	216	&	76.39	&	\textbf{76.85}	&	75.93	\\
heart	&	81	&	189	&	\textbf{85.19}	&	\textbf{85.19}	&	\textbf{85.19}	\\
heart	&	81	&	189	&	\textbf{85.71}	&	85.19	&	85.19	\\
heart	&	81	&	189	&	82.54	&	\textbf{84.13}	&	83.07	\\
\midrule
Average	&		&		&	80.33	&	\textbf{81.28}	&	80.52	\\
\bottomrule
\end{tabular} 
\end{table}

Even when including the balancing constraint, the global minimum may not improve on the (supervised) baseline. {This may happen for several reasons: an incorrect choice of hyperparameters, the clustering assumption not being verified, or the unlabeled data points not providing additional information compared to the labeled ones.} Failure of the clustering assumption is our conjecture on the \texttt{connectionist} dataset, where we also tested different sets of hyperparameters, managing to improve the accuracy of the baseline only very rarely. Different is the case of the four instances discussed in the previous section, namely two instances of \texttt{arrhythmia} with 10\% and 20\% of labeled data points, and two instances of \texttt{musk} with 10\% of labeled data points. In all four instances, the linear kernel was chosen by the cross-validation procedure. We argue that the reason for the poor results was the wrong kernel choice, which impacts the lower bound quality. Indeed, we solve these four instances by imposing the RBF kernel and setting $C_l=1$ (which is a standard value). The results are reported in Table~\ref{tab:special}, where we also include the lines from Table~\ref{tab:large_full} to facilitate the comparison. It turns out that the gaps are much smaller even though we cannot close them in three hours. These results confirm the theoretical result that the quality of our bounding routine is related to the kernel choice. Note also that, apart from the first instance, the final accuracy on the unlabeled data points is higher in this new setting, both for us and for the baseline, confirming the better quality of the kernel.

\begin{table}[!ht]
\footnotesize
\caption{In this table, we show the performance of our solver on four problematic instances by changing the kernel. We report the instance name, the number of labeled data points $l$, the number of unlabeled data points $n-l$, the kernel, the penalty parameter $C_l$, the percentage gap (Gap [\%]), the number of nodes (Nodes), the computational time in seconds (Time [s]), the accuracy of the produced solution on the unlabeled data points (Acc. [\%]), the accuracy of a supervised SVM built on the labeled data points and tested on the unlabeled ones (Base [\%]). We highlight in boldface the labeling that achieves the highest accuracy on the unlabeled data points.}\label{tab:special}
\begin{tabular}{lcc|cc|cccc|c}
\toprule
Instance	&	$l$	&	$n-l$	& Kernel & $C_l$ & Gap [\%] & Nodes & Time [s] & Acc. [\%] & Base [\%]\\  
\midrule
arrhythmia 	&	44	&	408	&	 linear 	&	0.10	&	9.50	&	563	&	10800	&	69.12	&	\textbf{69.61}	\\
arrhythmia 	&	44	&	408	&	 RBF 	&	1	&	0.79	&	453	&	10800	&	\textbf{68.87}	&	 66.67	\\
arrhythmia 	&	90	&	362	&	 linear 	&	0.32	&	10.91	&	532	&	10800	&	\textbf{67.40}	&	67.13	\\
arrhythmia 	&	90	&	362	&	 RBF 	&	1	&	0.29	&	435	&	10800	&	\textbf{70.99}	&	 69.89	\\
\midrule
musk 	&	46	&	430	&	 linear 	&	0.79	&	30.75	&	323	&	10800	&	\textbf{70.23}	&	69.3	\\
musk 	&	46	&	430	&	 RBF	&	1	&	0.51	&	277	&	10800	&	\textbf{76.51}	&	 75.81	\\
musk 	&	46	&	430	&	 linear 	&	0.10	&	9.11	&	503	&	10800	&	\textbf{67.44}	&	66.51	\\
musk 	&	46	&	430	&	 RBF 	&	1	&	0.49	&	382	&	10800	&	\textbf{76.28}	&	 70.47	\\
\bottomrule
\end{tabular}
\end{table}
To conclude the discussion on the achieved accuracy results, {we calculate in Table~\ref{tab:avgacc} the average percentage accuracy over the three seeds reported by the S3VM (column Acc. [\%]) and the baseline accuracy (column Base [\%]) for all instances except {\tt connectionist}. Furthermore, in the column ``Impr [\%]", we report the difference between ``Acc. [\%]" and ``Base [\%]", measuring the improvement (if positive) in the accuracy of the S3VM over the baseline supervised SVM. The average results confirm that the S3VM generally improves the average accuracy over the baseline supervised SVM. Increasing the number of labeled data points also enhances accuracy for both the S3VM and the baseline SVM. For some datasets, the S3VM demonstrates more significant improvements with fewer labeled data points, affirming that semi-supervised learning is particularly advantageous when labeled data is scarce. However, this increase in accuracy comes with the trade-off of significantly higher computational times, as we are solving a challenging global optimization problem.}
\begin{table}[!ht]
\caption{{In this table, we show the average accuracy across the three seeds. We report the instance name, the number of labeled data points $l$, the number of unlabeled data points $n - l$, the average S3VM accuracy (Acc. [\%]), and the average supervised SVM accuracy (Base [\%]). Additionally, we report in the column ``Impr [\%]" the difference between the columns ``Acc. [\%]" and ``Base [\%]".}}\label{tab:avgacc}
\begin{tabular}{lcc|cc|c}
\toprule
Instance	&	$l$	&	$n-l$	& Acc. [\%]  &  Base [\%] & Impr [\%]\\  
\midrule
art100  	&	10	&	90	&	93.70	&	92.96	&	0.74	\\
art100  	&	20	&	80	&	94.58	&	91.67	&	2.92	\\
art100  	&	30	&	70	&	94.76	&	90.00	&	4.76	\\
art150  	&	14	&	136	&	87.74	&	87.25	&	0.49	\\
art150  	&	29	&	121	&	88.43	&	87.60	&	0.83	\\
art150  	&	44	&	106	&	88.05	&	87.11	&	0.94	\\
heart   	&	27	&	243	&	77.37	&	76.68	&	0.69	\\
heart   	&	54	&	216	&	81.63	&	80.40	&	1.23	\\
heart   	&	81	&	189	&	84.84	&	84.48	&	0.35	\\
2moons  	&	30	&	270	&	92.22	&	92.10	&	0.12	\\
2moons  	&	60	&	240	&	99.31	&	99.31	&	0.00	\\
2moons  	&	90	&	210	&	99.21	&	99.21	&	0.00	\\
ionosphere      	&	34	&	317	&	89.17	&	84.96	&	4.21	\\
ionosphere      	&	70	&	281	&	90.27	&	87.54	&	2.73	\\
ionosphere      	&	104	&	247	&	89.34	&	89.34	&	0.00	\\
PowerCons       	&	36	&	324	&	96.09	&	94.86	&	1.23	\\
PowerCons       	&	72	&	288	&	96.99	&	96.64	&	0.35	\\
PowerCons       	&	108	&	252	&	98.68	&	99.08	&	-0.40	\\
GunPoint        	&	44	&	407	&	86.89	&	87.72	&	-0.82	\\
GunPoint        	&	89	&	362	&	94.57	&	94.47	&	0.09	\\
GunPoint        	&	134	&	317	&	93.80	&	94.43	&	-0.63	\\
arrhythmia      	&	44	&	408	&	71.08	&	65.44	&	5.63	\\
arrhythmia      	&	90	&	362	&	72.10	&	70.53	&	1.56	\\
arrhythmia      	&	135	&	317	&	73.92	&	73.82	&	0.11	\\
musk    	&	46	&	430	&	73.72	&	72.87	&	0.85	\\
musk    	&	94	&	382	&	80.37	&	80.37	&	0.00	\\
musk    	&	142	&	334	&	87.02	&	85.33	&	1.69	\\
wdbc    	&	56	&	513	&	95.97	&	95.71	&	0.26	\\
wdbc    	&	113	&	456	&	95.83	&	95.98	&	-0.14	\\
wdbc    	&	170	&	399	&	96.07	&	96.32	&	-0.25	\\
\bottomrule
\end{tabular}
\end{table}

Note that if a heuristic algorithm returns a low-quality solution, it is not possible to understand whether the solution is a bad local minimum or whether, for the given instance, the S3VM approach is not viable. On the other hand, we provide an exact method to directly establish the usefulness of the S3VM approach, at least for a given choice of hyperparameters. To the best of our knowledge, no other exact solution methods exist for real-world instances of such magnitude. 

The topic of this paper reinforces the cross-fertilization between mathematical optimization and machine learning, showing that optimization can be beneficial to machine learning (and vice versa) when competencies from both worlds are put together. From the optimization perspective, we have proposed an efficient method for a challenging problem. However, the correct interpretation of the results is only possible with a deep understanding of the semi-supervised learning problem. From the machine learning point of view, without the exact solver, it is impossible to decide whether the semi-supervised approach is suitable for a given instance of the problem.

\section{Conclusions}\label{sec:conclusion}

We proposed an SDP-based branch-and-cut algorithm for solving S3VM instances to global optimality. Similarly to the related literature, we considered the S3VM problem as a non-convex QCQP. Then, we applied optimality-based bound tightening to obtain box constraints. These constraints allowed us to include valid inequalities, strengthening the lower bound. For the lower bounding routine, we designed a cutting-plane algorithm. The resulting SDP relaxation yields bounds significantly stronger than those provided by convex relaxations available in the literature. For the upper bound, instead, we proposed a two-opt local search heuristic that exploits the solution of the SDP relaxation solved at each node. Computational results underscore the efficiency of our algorithm, demonstrating its capability to address instances with ten times more data points than methods currently available in the literature. To the best of our knowledge, no other exact solution method can effectively handle real-world instances of such magnitude. {Clearly, achieving exact solutions comes at the cost of significantly higher computational time compared to traditional supervised SVMs, which is expected due to the non-convex nature of the optimization problem.} This research provides insights from both optimization and machine learning perspectives. Firstly, solving large instances allows an understanding of when S3VMs are effective for a specific problem. Secondly, certified optimal solutions are valuable benchmarking tools for evaluating, enhancing, and developing heuristics and approximation methods. Such methods are fundamental in the machine learning literature. Thirdly, we assessed both theoretically and numerically the challenges that may arise when solving instances to global optimality and achieving high accuracy. The bottleneck in our algorithm lies in the lower bound computation, where we employed an off-the-shelf interior-point solver. For future research, we aim to design a tailored Alternating Direction Method of Multipliers (ADMM) algorithm capable of handling significantly larger sizes (more than a thousand data points) and a substantial number of valid inequalities. This should make the solution of the SDPs faster and more robust, leading to a more efficient branch-and-cut algorithm.

\section*{Acknowledgements}
Veronica Piccialli has been supported by PNRR MUR project PE0000013-FAIR. This research was funded in part by the Austrian Science Fund (FWF) [10.55776/DOC78]. For open access purposes, the author has applied a CC BY public copyright license to any author-accepted manuscript version arising from this submission.

\section*{Statements and Declarations}

\subsection*{Competing Interests}
The authors have no relevant financial or non-financial interests to disclose.

\bibliography{abbr, references}

\end{document}